\documentclass[10pt)]{amsart}  
\usepackage[english]{babel}
\usepackage[utf8]{inputenc} 
\usepackage[T1]{fontenc} 
\usepackage{amsfonts, amsmath, amssymb, amsthm, mathabx, mathtools} 
\usepackage{dsfont}
\usepackage{float}
\usepackage{color, xcolor}
\usepackage{graphicx}
\usepackage{tikz}
\usepackage{hyperref}
\usepackage{enumitem} 
\usepackage{pstricks-add}
\usepackage{palatino}

\definecolor{ref}{rgb}{0.1,0.2,0.7}

\newtheorem{theorem}{Theorem}
\newtheorem{corollary}{Corollary}
\newtheorem{lemma}{Lemma}
\newtheorem{remark}{Remark}

\newtheorem{definition}{Definition}
\newtheorem{proposition}{Proposition}
\newtheorem{assumption}{Assumption}

\def\barroman#1{\sbox0{#1}\dimen0=\dimexpr\wd0+1pt\relax
  \makebox[\dimen0]{\rlap{\vrule width\dimen0 height 0.06ex depth 0.06ex}%
    \rlap{\vrule width\dimen0 height\dimexpr\ht0+0.03ex\relax 
            depth\dimexpr-\ht0+0.09ex\relax}%
    \kern.5pt#1\kern.5pt}}

\newcommand{\R}{\mathbb{R}}
\newcommand{\E}{\mathbb{E}}
\newcommand{\p}{\mathbb{P}}

\DeclareMathOperator*{\argmin}{arg\,min}

\newcommand{\bs}[1]{\boldsymbol{#1}}
\newcommand{\trash}[1]{}
\title[McKean-Vlasov Equations with common noise]{A note on the long time behaviour of stochastic McKean-Vlasov equations with common noise}
\date{\today}

\address{Raphaël Maillet, Universit\'e Paris-Dauphine \& PSL, CNRS, CEREMADE, 75016 Paris, France}
\email{maillet@ceremade.dauphine.fr}
\author{Raphael Maillet}

\begin{document}
	\maketitle
		
		\begin{abstract}

		This paper focuses on the long-term behavior of solutions to nonlinear stochastic Fokker-Planck equations driven by common noise, where the drift term has a linear dependence on the measure. These equations, which describe the evolution of probability distributions, naturally arise in the mean-field limit of interacting particle systems driven by both idiosyncratic and common noises. After proving the existence of an invariant measure under some mild conditions, we first consider the case where the confinement potential is uniformly convex. In this setting, we establish a result of uniform-in-time conditional propagation of chaos for the associated particle system. This result directly implies the uniqueness of the long-term behavior for solutions of the nonlinear stochastic Fokker-Planck equation.
		Then, we highlight a more surprising phenomenon of uniqueness recovery induced by the addition of common noise in the non-convex case, albeit under more restrictive structural assumptions. Specifically, we show that the presence of common noise leads to uniqueness and exponential convergence towards equilibrium in the absence of idiosyncratic noise. This result emphasizes the stabilizing role of common noise in systems where non-convex potentials would typically allow for multiple invariant measures.
\color{black}
		\end{abstract}

\noindent \textbf{Mathematics Subject Classification (2010)}: 
{MSC}

\noindent \textbf{Keywords}: 
{Non-linear McKean-Vlasov, Stochastic Partial Differential Equations, Interacting particle system, Common noise, Conditional propagation of chaos, Asymptotic stability.}
		
\tableofcontents
\section{Introduction}
		
We consider the following non-linear Stochastic Partial Differential Equation (SPDE) on $[0, +\infty)\times\R^d$,
\begin{equation}\label{eq:main}
	\mathrm{d}_t  m_t = \nabla \cdot \left(\frac{\sigma^2 + \sigma_0^2}{2} \nabla m_t + m_t b(t, \cdot, m_t)\right)\mathrm{d}t - {\sigma_0 \nabla m_t \cdot \mathrm{d}B^0_t}.
\end{equation}
This SPDE is posed on a filtered probability space $\left(\Omega_0, \mathcal{F}^0, \mathbb{F}^0, \mathbb{P}_0\right)$, $B^0$ is a $d$-dimensional $\mathbb{F}^0$-Brownian motion, the drift $b : [0, +\infty)\times\R^d\times\mathcal{P}\left(\R^d\right) \to \R^d$ depends on time, space and measure, $\sigma$ and $\sigma_0$ are two non-negative constants. This paper aims to study the long-time behavior of solutions of Equation \eqref{eq:main}, and more precisely, to understand the effect of the common noise on the asymptotic stability. We assume that the drift term $b$ has a specific linear structure for the measure variable, with two continuously differentiable functions $V$ and $W$ such that $b(t,x,\mu) = -\nabla V(x) -\nabla W\ast\mu(x)$, where $\ast$ stands for the convolution operator. This assumption is typical when studying long-time behavior in McKean-Vlasov type equations, see $e.g$ \cite{bashiri2020long, bolley2013uniform, delmoral2019uniform, herrmann2010non, tugaut2013convergence}. In the following, we are interested in getting existence and uniqueness results for the invariant measure of the probability measure-valued process $(m_t)_{t\geq 0}$. As the problem under consideration falls within the realm of the McKean-Vlasov type, one may expect the existence of an invariant measure and uniqueness, at least in some specific cases. In this paper, we verify whether, in the case where $W$ is convex and $V$ is uniformly convex, the introduction of common noise does not compromise the classical uniqueness results of \cite{benachour1998nonlinear}. When $V$ is not convex, the matter becomes considerably more intricate, studied so far without the presence of common noise, only partial results are known. There exist cases in which the uniqueness of the invariant measure is not satisfied. Unlike linear elliptic equations, the presence of nonlinearity leads to the existence of multiple invariant measures. Specifically, it has been proven in \cite{herrmann2010non}, following the ideas of \cite{dawson1995stochastic} that when the confinement potential is uniformly convex outside of a ball centered in the origin, admits a double-well, and the diffusion coefficient $\sigma$ is sufficiently small, there exist exactly three invariant solutions of the following equation:
\begin{align*}
	\partial_t m_t = \frac{\sigma^2}{2} \Delta m_t + \nabla \cdot(m_t(\nabla V + \nabla W\ast m_t)).
\end{align*}
Since 2019, several papers \cite{delarue2019restoring}, \cite{delarue2020selection} or \cite{delarue2021exploration} investigate the restoration of uniqueness in mean-field games derived from deterministic differential games with a large number of players by introducing an external noise. In a similar manner, this paper explores the restoration of the uniqueness of the invariant measure by introducing common noise to the system. More precisely, we prove existence and uniqueness of the invariant measure for process $(m_t)_{t\geq 0}$ in the following cases: (1) when the confinement potential $V$ is uniformly convex and the interaction potential is convex. In order to achieve such result, we prove something stronger, mainly uniform-in-time conditional propagation of chaos for the approximating interacting particle system; (2) when the potential $V$ is not convex and there is no idiosyncratic noise in the system, $i.e$ $\sigma=0$. In this case, we also get exponential rates of convergence to the invariant measure. Formally, we show the existence of $\widebar P \in \mathcal{P}( \mathcal{P}(\R^d))$, and a constant $\eta >0$, such that for all initial condition $P_0$, we get the existence of a constant $C > 0$, such that for each time $t > 0$,
\begin{align*}
	\mathrm{d}_1^{\mathcal{P}\left( \R^d\right)}\left(P_t, \widebar P\right)\leq Ce^{-\eta t},
\end{align*}
where $P_t = \mathcal{L}(m_t)$, and $\mathrm{d}_1^{\mathcal{P}\left( \R^d\right)}(\cdot, \cdot)$ stands for the Wasserstein distance. 
\\

This paper presents the first results about existence and uniqueness of the stationary solution of the stochastic non-linear Fokker-Planck Equation with additive common noise. However, obtaining uniqueness of the invariant measure in the general case of Equation \eqref{eq:main} without strong convexity assumptions on the confinement potential appears to be a challenging problem that is not completely solved in this paper. The results obtained in the present paper are a step towards understanding the long-time behavior of the solutions to Equation \eqref{eq:main}, but are far from providing a complete understanding of the asymptotic stability for that kind of equation.
~\\

\noindent \textbf{Probabilistic setting \& Motivation.} Let us consider a filtered probability space $$(\Omega^1, \mathcal{F}^1, \mathbb{F}^1, \p^1).$$ Then, we define the following product structure
\begin{align*}
	\Omega = \Omega^0\times \Omega^1, \:\: \mathcal{F}, \:\: \mathbb{F}, \:\: \p,
\end{align*}
where  $(\mathcal{F}, \p)$ is the completion of the set $(\mathcal{F}^0 \otimes \mathcal{F}^1, \p^0\otimes \p^1)$ and $\mathbb{F}$ is the right continuous augmentation of $( \mathcal{F}^0_t \otimes \mathcal{F}^1_t)_{t \geq 0}$. We also consider a $d$-dimensional Brownian motion $B^0$ supported by $(\Omega^0, \mathcal{F}^0, \p^0)$, adapted to $\mathbb{F}^0$ and another Brownian motion $B$ supported by $(\Omega^1, \mathcal{F}^1, \p^1)$, adapted to $\mathbb{F}^1$ and independent of $\mathbb{F}^0$. { More explicitly, $B^0$ is a $\mathbb{F}^0$-Brownian motion, independent of $B$ which is a $\mathbb{F}^1$-Brownian motion. }
	Let us now consider a probability measure on the space of probability measures $P_0 \in \mathcal{P}(\mathcal{P}(\R^d))$, we are able to define 
$m_0$, a $\mathcal{F}^0_0$-measurable random variable with value in the space of probability measures $\mathcal{P}(\R^d)$ and such that $\mathcal{L}(m_0) = P_0$, in the sense that for any bounded measurable function $F : \mathcal{P}(\R^d) \to \R$, $\: \E_{\mathbb{P}^0}[F(m_0)] = \langle P_0 ; F \rangle$. We can now define on the whole probability space $(\Omega, \mathcal{F}, \mathbb{F}, \mathbb{P})$ a random variable $X_0$ such that $\mathcal{L}( X_0| \mathcal{F}^0_0) = m_0$ almost surely. Let us define the stochastic process $X$ evolving in $\R^d$, supported by $(\Omega, \mathcal{F}, \mathbb{F}, \mathbb{P})$, which dynamic is given by 
\begin{equation}\label{LP}
\left\{ \begin{array}{ll}
	dX_t = -\nabla V(X_t)\mathrm{d}t - \nabla W \ast m_t(X_t) \mathrm{d}t + \sigma \mathrm{d}B_t + \sigma_0 \mathrm{d} B^0_t\\
	X_{| t =0} = X_0,
\end{array} \right.
\end{equation}
where $m_t$ stands for the conditional law of the random variable $X_t$, with respect to the $\sigma$-algebra $\mathcal{F}^0_t$. Precisely, $m_t = \mathcal{L}\left( X_t | \mathcal{F}^0_t\right)$ almost surely, and $B$ is a $d$-dimensional $\mathbb{F}^1$-Brownian motion independent of $\mathbb{F}^0$. The dynamic of the process $(m_t)_{t\geq 0}$ is well known and given by the following Lemma (see for example \cite[Vol. II]{carmona2018probabilistic} among other references).

\begin{lemma}\label{L1}
    The measure-valued process $(m_t)_{t\geq 0}$ is a solution in the weak sense of the following Stochastic Partial Differential Equation:
\begin{equation}\nonumber
		\mathrm{d}_t  m_t = \nabla \cdot \left(\frac{\sigma^2 + \sigma_0^2}{2} \nabla m_t + m_t \left( \nabla V + \nabla W \ast m_t\right)\right)\mathrm{d}t -\sigma_0 \nabla m_t \cdot \mathrm{d}B^0_t,
\end{equation}
with initial condition $\mathcal{L}\left( m_0\right) = P_0$. 
\end{lemma}

{Equation \eqref{eq:main} connects closely to the McKean-Vlasov Equation with common noise. It shows how systems of $N$-interacting particles evolve in the mean-field limit $(N \to +\infty)$. In mathematical finance, this model is particularly useful for situations like inter-bank borrowing and lending systems (see \cite{carmona2013mean} or \cite{giesecke2015large}). Studying the long-term behavior of solutions to Equation \eqref{eq:main} is important because it gives us insight of the behavior of the solution to Equation \eqref{LP}. Such models also arise in applications such as population displacements, including bird flocking, where the common noise is interpreted as environmental noise (see, e.g., \cite{coghi2016propagation, choi2019cucker}).}

~\\

\noindent \textbf{Literature.} Stochastic Partial Differential Equations in the more general form,
\begin{align}\label{eq:eqgen}
\mathrm{d}_t m_t = \Big[ \sum_{i,j} \nabla^2_{ij}(a_{ij}(t, \cdot, m_t)m_t) +{\rm div} \left( m_t b(t, \cdot, m_t)\right)\Big]\mathrm{d}t -\sigma_0(t,\cdot,m_t) \nabla m_t \cdot \mathrm{d}B^0_t,
\end{align}
have been extensively studied in recent decades as they naturally arise in several applications. Equation \eqref{eq:eqgen} is linked to the stochastic scalar conservations law of the form
\begin{align}\label{eq:strato}
	\mathrm{d}_t m_t + \nabla\cdot(\sigma_0(\cdot, u_t)u_t)\circ \mathrm{d}W_t = 0,
\end{align}
where $\circ$ stands for the Stratonovich stochastic integral. In the case where $\sigma_0(x,\mu) = \sigma_0(\mu(x))$, meaning that the diffusion coefficient depends in the measure in a local way, this class of equations has been introduced in \cite{lions2013scalar} paving the way to several papers dealing with well-posedness of solutions of \eqref{eq:strato} in various frameworks \cite{lions2014scalar, gess2014scalar, friz2016stochastic, gess2017stochastic, fehrman2019well}. Uniqueness of the solutions to \eqref{eq:main} is a well-known result in the class of solutions admitting a square integrale density with respect to the Lesbegue measure, see \cite{kurtz1999particle}, and has been shown recently without any further moments assumptions in \cite{coghi2019stochastic}.

 In a slightly different context, a series of papers demonstrated the well-posedness for a large class of Stochastic Differential Equations similar to \eqref{eq:main}, called \textit{Mean Reflected Stochastic Differential Equations}, see \cite{briand2018bsdes}, \cite{briand2021particles} and \cite{briand2020forward}. More precisely, \cite{briand2021particles} and \cite{briand2020forward} state conditional propagation of chaos under regularity conditions on the drift and diffusion terms. These equations naturally appear when considering interacting particle systems with constraints on the empirical measure of the systems and then the study of such equations is particularly important for example for applications to Mean Field Games.  

In this paper, we focus on the following specific stochastic non-linear Fokker-Planck Equation with common noise:
\begin{align}\label{eq:SMKVCN}
	\mathrm{d}_t  m_t = \nabla \cdot \left(\frac{\sigma_0^2 + \sigma^2}{2} \nabla m_t + m_t b(t, \cdot, m_t)\right)\mathrm{d}t  -\sigma_0 \nabla m_t \cdot \mathrm{d}B^0_t.
\end{align}
 
As mentioned before, extensive research has been conducted on the equation in question, and recent studies have made notable contributions to understanding its properties. For example \cite{hammersley2021weak} explores the existence and uniqueness of solutions for McKean-Vlasov Stochastic Differential Equations (SDEs) with common noise. Similarly, \cite{marx2021infinite} proposes a regularization approach in an infinite-dimensional setting for the McKean-Vlasov equation with Wasserstein diffusion, enhancing the understanding of solutions' regularity properties. Additionally, \cite{kumar2022well} investigates the well-posedness and numerical methods for McKean-Vlasov equations with common noise, providing valuable insights on the stability and convergence of computational approaches for solving these equations.

{ To the best of our knowledge, the asymptotic behavior of the measure-valued flow $(m_t)_{t\geq 0}$ solution to equation  \eqref{eq:SMKVCN} remains largely unexplored, even in cases where the drift \( b \) is linear with respect to the measure variable}. In the case without common noise, $\sigma_0=0$, where $m$ is a deterministic flow of measures, past research has focused on various aspects of the solutions of \eqref{LP}, including existence, uniqueness (\cite{mckean1966class}, \cite{funaki1984certain}, \cite{graham1996asymptotic}), and stability. Over the past two decades, significant advancements have been made in understanding the convergence to equilibrium for solutions of the deterministic McKean-Vlasov equation. For example, see \cite{carrillo2003kinetic} or \cite{carrillo2006contractions} for proofs of an exponential convergence rate to equilibrium under strict convexity conditions on the potentials $V$ and $W$. The case without strict convexity assumptions is more intricate. Nevertheless, through a thorough examination of the dissipation of the Wasserstein distance, \cite{bolley2013uniform} showed an exponentially fast convergence to equilibrium in a weakly convex case. Recently, involving a coupling method issued from \cite{lindvall1986coupling}, it has been shown using nice concentration properties from \cite{eberle2016reflection} that the convergence to equilibrium holds with an exponential speed in the case of a confinement potential that is only convex far from the origin, as seen in \cite{durmus2020elementary}. The latter shows uniform in time propagation of chaos property, as introduced in \cite{kac1956foundations} and \cite{sznitman1991topics}, allowing one to conclude the uniqueness of the invariant measure and provide a rate of convergence to equilibrium. 
~\\

\noindent \textbf{Organisation of the paper.} This paper has three main parts. In Section \ref{sec:se2}, we study the existence of an invariant measure for the process $(m_t)_{t \geq 0}$, with dynamic given by Equation \eqref{eq:main}, and provide conditions on the potentials $V$ and $W$ for the existence of such an invariant measure. We also give moment estimates for the invariant measures. In Section \ref{se:se3}, we study the uniqueness of invariant measures with a uniformly convex confinement potential $V$, adapting known results without common noise. We also discuss the Ornstein-Uhlenbeck process with common noise, where the invariant measure can be explicitly described. Moreover, we demonstrate uniform-in-time propagation of chaos and convergence to equilibrium. Section \ref{sec:se4} explores the same topic for non-convex potential $V$ when $\sigma = 0$ and we prove that common noise can help to restore the uniqueness of the invariant measure. Technical proofs are provided in Appendix \ref{Appendix}.
~\\

\noindent \textbf{Definition and Notation.} Throughout the paper, for a Polish space $E$ we write $\mathcal{P}(E)$ for the space of Borel probability measures on $E$ equipped with the topology of weak convergence and the corresponding Borel $\sigma$-algebra. We also denote by $\langle \cdot \: ; \: \cdot\rangle$ the duality product on $\mathcal{P}(E)$.  In this paper, we consider a stochastic process $(m_t)_{t\geq 0}$ with value in the space of probability measures $\mathcal{P}\left(\R^d\right)$. We denote by $P_t$ the law $\mathcal{L}\left(m_t\right)$ of $m_t$, for $t \geq 0$, which is a probability measure on the space of probability measures. Then $(m_t)_{t \geq 0}$ is a continuous $\mathcal{P}\left(\mathbb{R}^d\right)$-valued process, and $P=(P_t)_{t \geq 0}$ belongs to $C\left([0,+\infty[ ; \mathcal{P}\left(\mathcal{P}\left(\mathbb{R}^d\right)\right)\right)$, the space of continuous functions from $[0, +\infty[$ to $\mathcal{P}\left(\mathcal{P}\left(\mathbb{R}^d\right)\right)$. 

{Whenever there exists a distance $d$ so that $(E,d)$ is a metric space, 
we call ${\mathcal P}_p(E)$, for any $p>0$, the collection of elements $\mu \in {\mathcal P}_p(E)$ such that 
\begin{equation*}
	\exists x_0 \in E : \int_E d(x_0,x)^p \mu(\mathrm{d}x) < +\infty.
\end{equation*}
In fact, the integral above is finite or not, whatever the choice of $x_0$. We then} define for any $p\geq 1$, the $p$-Wasserstein distance on 
$\mathcal{P}_p(E)$ as
\begin{equation*}
	\mathrm{d}_p^E(\mu, \nu) := \inf_{\pi \in \Pi(\mu, \nu)} \left(\int_{E\times E} d(x,y)^p \: \pi(\mathrm{d}x, \mathrm{d}y)\right)^{1/p},
	\quad \mu,\nu \in {\mathcal P}_p(E), 
\end{equation*}
where $\Pi(\mu, \nu)$ stands for the set of all couplings of $\mu$ and $\nu$ (i.e., all the joint
probability measures on $E \times E$ with $\mu$ and $\nu$ as first and second marginals). Also for any probability measure $\mu \in \mathcal{P}(\R^d)$ and $p \geq 1$ we denote by $L^p(\mu)$ the set of functions $f : \R^d \to \R^d$ such that
\begin{equation*}
	\| f \|_{L^p(\mu)}^p := \int_{\R^d} |f(x)|^p \mu(\mathrm{d}x) < +\infty. 
\end{equation*} 

In this paper, we mainly use a probability-based approach, often switching between the measure-valued stochastic process $m$ and its probabilistic counterpart $X$, which solves Equation $\eqref{LP}$. At this point, it is worth noting that we are concerned with a stochastic process $(m_t)_{t\geq 0}$ that takes values in the space of probability measures $\mathcal{P}(\R^d)$. This means that at each time $t >0$, we are dealing with measures on the space of probability, rather than on the underlying space $\R^d$. We define the notion of an invariant measure in $\mathcal{P}(\mathcal{P}(\R^d))$, which is a probability measure that remains invariant under the evolution of the stochastic process. These definitions are essential for our analysis and are detailed below.
\begin{definition}\label{D1}
	Let us consider a random variable $X$ defined on the filtered probability space $( \Omega, \mathcal{F},\mathbb{F}, \mathbb{P})$, and $P \in \mathcal{P}_2(\mathcal{P}_2(\R^d))$. According to Lemma 2.4 in \cite{carmona2018probabilistic}, for $\mathbb{P}^0-a.e.$ $\omega_0\in\Omega^0$, $X(\omega_0, \cdot)$ is a random variable on $(\Omega^1, \mathcal{F}^1, \mathbb{P}^1)$. By defining $\mathcal{L}^1(X) : \Omega^0 \ni \omega_0 \mapsto \mathcal{L}(X(\omega_0, \cdot))$, we get a random variable from $(\Omega^0, \mathcal{F}^0, \mathbb{P}^0)$ into $\mathcal{P}(\R^d)$, providing a conditional law of $X$ given $\mathcal{F}^0$. Finally, we say that $\mathcal{L}( X ) = P$ whenever $\mathcal{L}^1(X)$ is distributed with respect to $P$.
\end{definition}

\begin{definition}[Invariant measure]\label{D2}
We say that $\widebar P \in \mathcal{P}_{2}(\mathcal{P}_{{2}}(\R^d))$ {is} an invariant measure for the process $(m_t)_{t\geq 0}$ {when the latter is regarded as taking values in ${\mathcal P}_2({\mathbb R}^d)$}, if the law of $m_t$ is independent of $t$, when $m_0$ is distributed according to $\widebar P$, 
{i.e.}, for all continuous and bounded function $\phi \in \mathcal{C}_b(\mathcal{P}_{{2}}(\R^d))$
	\begin{align*}
		 {\E_0[\phi(m_t)]} = \E_0[\phi(m_0)] = \int_{\mathcal{P}_{{2}}(\R^d)} \phi(m) \widebar P(\mathrm{d}m), \quad \forall t > 0. 
	\end{align*}
Moreover, in the following, we say (with a slight abuse of notation) that a stochastic process $X$ on the probability space $\left(\Omega, \mathcal{F}, \mathbb{F} = \left(\mathcal{F}_t\right)_{t\geq 0},  \mathbb{P}\right)$ admits an invariant measure in $\mathcal{P}_2(\mathcal{P}_2(\R^d))$ if and only if, the measure-valued stochastic process $m_t = \mathcal{L}\left(X_t | \mathcal{F}^0_t\right)$ admits an invariant measure.
\end{definition}
Of course, if $\bar P$ is an invariant probability distribution, then the process $(m_t)_{t\geq 0}$ with initial condition $\bar P$, has the same law as the process  $(m_{t+T})_{t\geq 0}$ for any $T > 0$.  This follows directly from the fact that $m_0$ and $m_T$ have the same law combined with the weak Markov property.
\color{black}

\begin{definition}\label{D4}
	For any function $f$, with at most linear growth, such that $(x,y) \mapsto f(|x-y|)$ defines a distance on $\R^d$, we define the following distance on $\mathcal{P}_1(\R^d)$
	\begin{equation}
		\mathrm{d}_f^{\R^d}(\mu, \nu) := \inf_{\pi \in \Pi(\mu, \nu)} \left(\int_{\R^d\times \R^d} f(|x-y|)) \: \pi(\mathrm{d}x, \mathrm{d}y)\right).
	\end{equation}
Moreover, we define on $\mathcal{P}_1(\mathcal{P}_1(\R^d))$, 
	\begin{equation}\nonumber
		\mathrm{d}_f^{\mathcal{P}(\R)}( P,Q) = \inf_{\Gamma \in \Pi(P,Q)} \int_{\mathcal{P}_1(\R^d)} \mathrm{d}_f^{\R^d}( \mu, \nu) \Gamma( \mathrm{d}\mu, \mathrm{d}\nu),
	\end{equation}
	for any $P,Q \in \mathcal{P}_1(\mathcal{P}_1(\R^d))$. 
\end{definition}

\begin{remark}
We will frequently use the fact that $\mathrm{d}^E_f$ is a deformation of the Wasserstein distance, while emphasizing that these two distances are equivalent. The equivalence of these distances will play a key role in our analysis. For further details on the introduction and use of such deformed distances, we refer the reader to the work of Eberle \cite{eberle2016reflection}.
\end{remark}

\noindent \textbf{Our contribution.} { Let us recall that in this paper, we ony focus on drifts of the type 
\begin{equation*}
	b(x,\mu) = -\nabla V(x) -\nabla W \ast \mu\left(x\right), \quad x \in {\mathbb R}^d, \ \mu \in {\mathcal P}(\R^d), 
\end{equation*}
for $V$, $W : \R^d \to \R$. This work aims to establish a framework for analyzing the long-time behavior of solutions to \eqref{eq:main}. Specifically, we demonstrate the existence of invariant measures for the process $(m_t)_{t \geq 0}$, governed by the dynamics \eqref{eq:main}, under relatively mild conditions. Furthermore, we show that when the confining potential $ V$ satisfies strong conditions (uniform convexity), we can obtain the uniqueness of the invariant measure. The proof of this uniqueness result is done by establishing uniform-in-time conditional propagation of chaos for the natural particle approximation system, using the well-known method of synchronous coupling. This approach provides a first result on the uniqueness of the invariant measure for the solution to \eqref{eq:main}

A significant part of this paper focuses on the case where the potential is non-convex. The final contribution addresses the long-time behavior of systems with non-convex potentials (convex at infinity) and a linear interaction function. Without common noise, uniqueness of the invariant measure may fail. However, we prove that common noise, in the absence of idiosyncratic noise, restores uniqueness. We explicitly describe an invariant measure and show its uniqueness for any arbitrary small level of common noise. We also prove exponential convergence under sufficiently strong attractive forces between particles—a regime where uniqueness typically fails without common noise.

The proof uses reflection coupling, adapted to handle non-convexity and the finite-dimensional structure of the common noise. Similar results exist for nonlinear equations (see, e.g., \cite{angeli2023mckean}), but our case is more challenging due to the finite-dimensional structure of the common noise and the fact that the natural state space is infinite-dimensional. This result does not claim that finite-dimensional common noise universally restores uniqueness but provides examples where this phenomenon occurs.
}\color{black}

	\subsection{The process $m$ as a mean field limit}
	
Thanks to the definition of the previous subsection, we are now ready to give an interpretation of the process $(m_t)_{t\geq 0}$ in terms of mean-field limit for interacting particle system. Let us consider $P_0 \in \mathcal{P}_2(\mathcal{P}_2(\R^d))$, let also $N \geq 1$ be an integer and $( X_0^{1,N}, \dots, X_0^{N,N})$, $N$ random variables which are conditionally independent and identically distributed with respect to $\mathcal{F}^0_0$, such that $\mathcal{L}( X^{i,N}_0) = P_0$ for all $i \in \{ 1 , \dots , N\}$. We now define the following interacting particle system
	\begin{equation}\nonumber
		\left\{ \begin{array}{ll}
			\mathrm{d}X^{i,N}_t = -\nabla V(X^{i,N}_t) - N^{-1}\sum_{j=1}^N\nabla W( X^{i,N}_t - X^{j,N}_t) \mathrm{d}t + \sigma \mathrm{d}B^i_t + \sigma_0 \mathrm{d} B^0_t,\\
			X^{i,N}_{| t = 0} = X^{i,N}_0, \quad \forall i \in \left\{1, \dots,  N\right\},
		\end{array}\right.
	\end{equation}
where the $B^i$ are independent $d$-dimensional $\mathbb{F}^1$-Brownian motion which are independent of $\mathbb{F}^0$. Then, we consider the mean-field limit system $(\widebar X^1, \dots, \widebar X^N)$ driven by
\begin{equation}\nonumber
		\left\{ \begin{array}{ll}
			\mathrm{d}\widebar X^{i}_t = -\nabla V(\widebar X^{i}_t) - \nabla W\ast m_t\left( \widebar X^i_t\right) \mathrm{d}t + \sigma \mathrm{d}B^i_t + \sigma_0 \mathrm{d} B^0_t,\\
			\widebar X^{i}_{| t = 0} = \widebar X^{i}_0, \quad \forall i \in \left\{1, \dots,  N\right\},
		\end{array}\right.
	\end{equation}
where $\left(\widebar X_0^i\right)_i$ are conditionally independent and identically distributed (i.i.d) and random variables with respect to $\mathcal{F}^0_0$, such that $\mathcal{L}\left( \widebar X^i_0\right) = P_0$, for all $i \in \left\{ 1, \dots, N\right\}$. Our framework is exactly the same as the classical one for mean-field games system with common noise, see \cite[Vol. II]{carmona2018probabilistic}. However here, the law of the initial conditions is random. Then, conditioning concerning the $\sigma$-algebra $\mathcal{F}^0_0$, we get back to a more classical framework where the initial condition is a deterministic measure. More precisely, as stated in \cite[Vol. II]{carmona2018probabilistic}, under sufficient regularity conditions on the transport part $b$ mainly Lipschitz continuity with respect to the space and measure variables, we get that for any fixed $t \geq 0$,
\begin{equation}\nonumber
		\lim_{N \to +\infty} \E[ | \widebar X^i_t - X^{i,N}_t |^2 ] + \E[\mathrm{d}_2^{\R^d}( m^N_t, m_t)] = 0,
	\end{equation}
where $\mathrm{d}_2^{\R^d}(\cdot,\cdot)$ stands for the classical Wasserstein distance on $\R^d$, and $m^N_t := N^{-1}\sum_{i=1}^N \delta_{X^{i,N}_t}$ is the empirical measure of the interacting particle system. 

		%Existence of an invariant measure
		\section{Existence of an invariant measure for the stochastic flow of measures.}\label{sec:se2}
		
		Let us consider a stochastic process $(m_t)_{t\geq 0}$ with value in the space of probability measures $\mathcal{P}\left(\R^d\right)$, and with dynamic given by \eqref{eq:main}. Then, $m$ is a weak solution of 
\begin{align}
	\mathrm{d}_t  m_t = \nabla \cdot \Big(\frac{\sigma^2 + \sigma_0^2}{2} \nabla m_t + m_t (\nabla V + \nabla W \ast m_t)\Big)\mathrm{d}t -\sigma_0 \nabla m_t \cdot \mathrm{d}B^0_t.
\end{align}
More precisely, for all $t \geq 0$ and $\varphi \in C_c^{\infty}(\mathbb{R}^d)$,
\begin{align*}
\mathrm{d}\langle m_t, \varphi\rangle=\langle m_t, L_{m_t} \varphi \rangle \: \mathrm{d} t + \sigma_0\langle m_t,(\nabla \varphi)^{\top}\rangle \: \mathrm{d} B^0_t, 
\end{align*}
where for any probability measure $m \in \mathcal{P}(\R^d)$, the operator $L_{m}$ acts on a smooth function $\varphi$ of compact support by
$$
L_{m} \varphi = - (\nabla V  + \nabla W\ast m ) \cdot \nabla \varphi+\frac{\sigma_0^2+\sigma^2}{2} \Delta \varphi,
$$	
where $\nabla, \Delta$ respectively stands for the gradient and Laplacian operator, while $\cdot$ denotes the usual inner product in $\R^d$, and for any $\varphi \in C_c^{\infty}(\mathbb{R}^d)$ and any probability measure $m$,
\begin{align*}
    \langle m ; \varphi\rangle = \int_{\R^d} \varphi \:\mathrm{d}m.
\end{align*}  
In this section, we aim to give conditions on the potentials $V$ and $W$ to ensure the existence of an invariant measure for the process $(m_t)_{t\geq 0}$. Let us now consider the following assumptions:
{
\begin{assumption}\label{A1}
	The drift $V : \R^d \to \R$ is twice differentiable in $\R^d$. Moreover, 
	\begin{itemize}
	\item[\rm (A5.1)]$V$ is confining in the sense that there exists a function $\kappa : [0; +\infty) \to \R$ 
		 {such that} 
    \begin{equation}\nonumber  
        \left( \nabla V(x) - \nabla V(y)\right)\cdot\left(x-y\right) \geq \kappa(| x-y|) |x-y|^2.
    \end{equation}
		{with}
		\begin{equation*}
		\lim\sup_{r \to +\infty}\kappa(r) > 0 \quad 
		\textrm{\rm and} \quad 
			\int_0^1 r \kappa(r)^{-} \mathrm{d} r<\infty, 
		\end{equation*}
		where $\kappa^{-} = \max(0, -\kappa)$. 
	\item[\rm (A.5.2)] $\nabla V$ is Lipschitz continuous with constant $L_V > 0$. 
\end{itemize}
\end{assumption}
}

\begin{assumption}\label{A2}
	$W$ is twice differentiable. Moreover, %we assume that:
	\begin{itemize}
		\item[{\rm (A2.1)}] $W$ is symmetric, i.e., $W(x)=W(-x)$ for all $x \in \mathbb{R}^d$;
		\item[{\rm (A2.2)}] $\nabla W$ is $L_W$-Lipschitz continuous.	
	\end{itemize}
\end{assumption}

In the following, we work under the set of assumptions \ref{A1} \& \ref{A2}. The assumption regarding confinement potential $V$ primarily ensures convexity at infinity, which helps keep the process within a compact set with a high probability. We can moreover note that this implies the existence of $m_V > 0$ and $M_V \geq 0$, such that 
\begin{equation}\nonumber
    \left( \nabla V(x) - \nabla V(y)\right)\cdot\left(x-y\right) \geq m_V|x-y|^2 - M_V.
\end{equation}

{
The two following propositions stand. 
\begin{proposition} 
\label{prop:existence:!}
Under Assumptions~\ref{A1} and~\ref{A2} and for an initial condition $X_0$ satisfying 
${\mathbb E}[\vert X_0 \vert^2]< \infty$, 
the conditional McKean-Vlasov equation 
\eqref{LP} has a unique 
${\mathbb F}$-progressively measurable solution
$(X_t)_{t \geq 0}$, with continuous trajectories, such that, for 
all $T>0$, ${\mathbb E} [\sup_{0 \le t \le T} \vert X_t\vert^2] < \infty$. 
\end{proposition}

{\begin{proof}
The proof consists of a straightforward fixed point argument, using the Lipschitz properties of $\nabla V$ and $\nabla W$. We refer to~\cite[Vol. II, Chap. 2]{carmona2018probabilistic}. 
\end{proof}}
}

{By the superposition principle for conditional McKean-Vlasov equations (see 
\cite{lacker2020superposition}), we deduce that existence and uniqueness also hold true for the 
stochastic Fokker-Planck equation: 
\begin{proposition}
Let Assumptions~\ref{A1} and~\ref{A2} be in force. Then for an initial condition $m_0$ satisfying  
$${\mathbb E}_{\mathbb{P}^0} \left[\int_{{\mathbb R}^d} \vert x \vert m_0({\mathrm d}x) \right]< \infty,$$
the stochastic Fokker-Planck equation 
\eqref{eq:main} has a unique solution $(m_t)_{t \geq 0}$
in the space of ${\mathbb F}^0$-progressively measurable processes 
with values in ${\mathcal P}_2(\R^d)$. Moreover, this solution satisfies 
$${\mathbb E}_{\mathbb{P}^0}\left[\sup_{0 \leq t \leq T} \int_{{\mathbb R}^d} \vert x \vert ^2
m_t({\mathrm d}x)\right] < \infty.$$
\end{proposition} 
{
\begin{proof} 
The existence of a solution directly follows from the existence part of the statement in 
Proposition \ref{prop:existence:!}, which corresponds to Proposition 1.2 in 
\cite{lacker2020superposition}. However, establishing uniqueness is more challenging. To address this, we rely on Theorem 1.3 in 
\cite{lacker2020superposition}. In summary, any weak solution $(\widetilde m_t)_{t \geq 0}$ 
to \eqref{eq:main} induces a weak solution $ (\widetilde X_t)_{t \geq 0}$ 
to the McKean-Vlasov equation \eqref{LP}, where the idiosyncratic noise $\widetilde B$ becomes part of the solution. It holds that $\widetilde m_t = {\mathcal L}(\widetilde X_t \vert {\mathcal F}_T^0)$. Moreover, by the strong uniqueness of \eqref{LP}, a relevant form of the Yamada-Watanabe theorem applies, implying that 
${\mathcal L}(\widetilde X_t \vert {\mathcal F}_T^0)$ must equal 
${\mathcal L}_1(X_t)$, where $X = (X_t)_{t \geq 0}$ is the solution given by 
Proposition \ref{prop:existence:!}.
\end{proof} 
}
In this section, we start with a result regarding uniform-in-time control of the process $(X_t)_{t\geq 0}$ with initial condition in $\mathcal{P}_2(\mathcal{P}_2(\R^d))$ dynamic given by \eqref{LP}. To prove the existence of an invariant measure for $(m_t)_{t\geq 0}$ solution of \eqref{eq:main}, we will use the concept of intrinsic derivative for a functional defined on a space of measure, see for example \cite[Def 2.1]{cardaliaguet2019master}. 
{
\begin{definition}\label{D:diff} Throughout the definition, ${\mathcal P}_2({\mathbb R}^d)$ is equipped with $\mathrm{d}_2^{{\mathbb R}^d}$.
Let us define $\mathcal{C}_b^2(\mathcal{P}_2(\mathbb{R}^d))$ as the collection of continuous and bounded functions $\Phi: \mathcal{P}_2(\mathbb{R}^d) \to \mathbb{R}$ with the following properties:  
\begin{itemize}
    \item There exists a unique jointly continuous function $\partial_m \Phi: (m,x) \in \mathcal{P}(\mathbb{R}^d) \times \mathbb{R}^d \to \partial_m \Phi(m,x) \in \mathbb{R}^d$, at most of quadratic growth in $x$ for any $m$, such that 
    \[
    \lim _{h \to 0} \frac{\Phi(m+h(m^{\prime}-m))-\Phi(m)}{h}=\int_{\mathbb{R}^d} \partial_m \Phi(m, v)(m^{\prime}-m)(\mathrm{d} v),
    \]
    for all $m, m^{\prime} \in \mathcal{P}(\mathbb{R}^d)$ and
    \[
    \int_{\mathbb{R}^d} \partial_m \Phi(m, v) m(\mathrm{d} v)=0, \quad m \in \mathcal{P}(\mathbb{R}^d);
    \]

    \item For any $m \in {\mathcal P}_2({\mathbb R}^d)$, the mapping $x \mapsto \partial_m \Phi(m, x)$ is differentiable, with the gradient being denoted $D_m \Phi(m, x) := \nabla_x \partial_m \Phi(m, x) $; the mapping $(m,x) \mapsto D_{m} \Phi(m,x)$ is jointly continuous in $(m,x)$ and at most of linear growth in $x$ uniformly in $m$ in bounded subsets of ${\mathcal P}_2({\mathbb R}^d)$;

    \item For any $x \in \mathbb{R}^d$, every component of the $\mathbb{R}^d$-valued function $m \mapsto D_m \Phi\left(m, x\right)$ satisfies the same conditions as in the first bullet point, resulting in a mapping $(m,x,y) \in {\mathcal P}_2({\mathbb R}^d) \times {\mathbb R}^d \times {\mathbb R}^d \mapsto D_m^2 \Phi(m, x, y) \in \mathbb{R}^{d \times d}$, which is jointly continuous (in the three arguments) and at most of quadratic growth in $(x,y)$, uniformly in $m$ in bounded subsets of ${\mathcal P}_2({\mathbb R}^d)$;

    \item For any $m \in \mathcal{P}_2(\mathbb{R}^d)$, the function $x \in {\mathbb R}^d \mapsto D_m \Phi(m,x)$ is differentiable; the Jacobian, denoted $(m,x) \mapsto D^2_{xm} \Phi(m, x)$, is jointly continuous in $(m, x)$ and at most of linear growth in $x$, uniformly in $m$ in bounded subsets of ${\mathcal P}_2({\mathbb R}^d)$.
\end{itemize}
\end{definition}
}}

\begin{proposition}\label{Lacker}
Considering a measure valued process $(m_t)_{t\geq 0}$, which dynamic is given by \eqref{eq:main}, and defining $P_t = \mathcal{L}\left(m_t\right)$, we have that for any bounded and twice differentiable function $F \in \mathcal{C}_{b}^{2}(\mathcal{P}(\mathbb{R}^{d}), \R)
 $:
\begin{align}\label{eq:equationP}
	\left\langle P_{t}-P_{0}, F\right\rangle=\int_{0}^{t}\left\langle P_{s}, \mathcal{M} F\right\rangle \mathrm{d} s, \quad \forall t >0, 
\end{align}
where we define, for $m \in \mathcal{P}(\mathbb{R}^{d})$,
$$
\begin{aligned}
\mathcal{M} F(m):=& \int_{\mathbb{R}^{d}}\Big[D_{m} F(m, x) \cdot b(x, m)+\frac{\sigma^2+ \sigma_0^2}{2} \nabla\cdot (D_{m} F(m,x))\Big] m(dx) \\
&+\frac{\sigma_0^2}{2} \int_{\mathbb{R}^{2d}} \mathrm{Tr}\left[D_{mm}^{2} F\left(m, x, y\right)\right] m(dx) m(dy),
\end{aligned}
$$
and where for all measurable and bounded function $\Phi : \mathcal{P}(\R^d) \to \R$, and all $P \in \mathcal{P}(\mathcal{P}(\R^d))$, 
$$
\langle P;\Phi\rangle = \int_{\mathcal{P}\left( \R^d\right)} \Phi(m) P(\mathrm{d}m). 
$$ 
Moreover, $\widebar P$ is an invariant measure if and only if
\begin{align}\label{eq:InvMeasure}
	\left\langle \widebar P, \mathcal{M} F\right\rangle = 0, \qquad \forall F \in \mathcal{C}^2_b( \mathcal{P}(\R^d), \R). 
\end{align}
\end{proposition}
The proof of Proposition \ref{Lacker} is based on \cite[Section 1.2]{lacker2020superposition} and postponed to the Appendix \ref{App:A0}. 

	\subsection{The existence result} 
	
We now present a result about the existence of an invariant measure for the equation of interest, adapting classical results for the McKean-Vlasov equation without common noise where the flow of probability measures, represented as $(m_t)_{t\geq 0}$ is deterministic. 
 We begin this section with the following Lemma, the proof of which is classical and postponed to Appendix \ref{App:A1}.
\begin{lemma}\label{le:L2}
    Under Assumptions \ref{A1} \& \ref{A2}, let us consider $P_0 \in\mathcal{P}_2(\mathcal{P}_2(\R^d))$.
 Then, denoting by $(X_t)_t$ the associated stochastic process with initial condition $\mathcal{L}\left(X_0\right) = P_0$ in the sense of Definition \ref{D1}, and dynamic given by \eqref{LP}, we have the following uniform in time moment control:
\begin{equation}\nonumber
        \sup_{t > 0} \E\left[ |X_t|^2\right] < +\infty.
    \end{equation}
\end{lemma}
Now, we are ready to state the following Proposition:
\begin{proposition}\label{T1}
    Under Assumptions \ref{A1} \& \ref{A2}, the dynamical system given by \eqref{eq:main} admits at least an invariant measure $\widebar P \in \mathcal{P}_2(\mathcal{P}_2(\R^d))$, $i.e$ such that
	\begin{align*}
		\int_{\mathcal{P}\left( \mathbb{R}^d\right)}\int_{\mathbb{R}^d}|x|^2 m(\mathrm{d}x) \widebar P(\mathrm{d}m) < +\infty.
	\end{align*}
\end{proposition}

\begin{proof}[Proof of Proposition \ref{T1}]
Let us fix $P_0 \in \mathcal{P}_2(\mathcal{P}_2(\R^d))$, $i.e$ such that:
\begin{equation}\nonumber
	\int_{\mathcal{P}\left( \mathbb{R}^d\right)}\int_{\mathbb{R}^d}|x|^2 m(\mathrm{d}x) P_0(\mathrm{d}m) < +\infty.
\end{equation}

	\textit{Step 1}.   Let us now consider a random flow of probability measures $(m_t)_{t\geq 0}$ with dynamic given by \eqref{eq:main} and initial condition $P_0 \in \mathcal{P}_2(\mathcal{P}_2(\R^d))$. At each time $t \geq 0$, we denote by $P_t$ the law of this measure-valued process.  For $T > 0$, we define the process $\left( Q_T\right)_{T > 0}$ with value in  $\mathcal{P}_2(\mathcal{P}_2(\R^d))$ by
	\begin{equation}\nonumber
		Q_T = T^{-1}\int_0^T P_t \mathrm{d}t.
	\end{equation}
	This defines a sequence of probability measures on $\mathcal{P}( \mathbb{R}^d)$. Let us show that $(Q_T)_{T > 0}$ admits at least a convergent subsequence. Thanks to Prohorov theorem, we only need to show that this sequence is tight. For $R > 0$, let us consider the set 
\begin{equation}\nonumber
	K_R = \Big\{ m \in \mathcal{P}( \mathbb{R}^d), \: \int_{\mathbb{R}^d} |x|^2 m(\mathrm{d}x) \leq R\Big\}. 
\end{equation}
This set is compact for the topology of weak convergence in $\mathcal{P}( \mathbb{R}^d)$. 
Now, for $T > 0$,
\begin{equation}\nonumber
	\begin{aligned}[t]
		Q_T(K_R) &= T^{-1}\int_0^T P_t(K_R) \mathrm{d}t \\
		& = T^{-1}\int_0^T \int_{\mathcal{P}\left( \mathbb{R}^d\right)} \mathds{1}_{\left\{m \in K_R \right\}} P_t(\mathrm{d}m) \mathrm{d}t \\
		& = T^{-1}\int_0^T \E\left[ \mathds{1}_{\left\{m_t \in K_R \right\}}\right] \mathrm{d}t \\
		& \geq 1 - \frac{c}{R},
	\end{aligned}
\end{equation}
where $c = \sup_{t > 0} \E\left[ |X_t|^2\right] < + \infty$, thanks to Lemma \ref{le:L2}. Hence, for any $\varepsilon > 0$, there exists $R_{\varepsilon} > 0$ such that $Q_T\left(K_{R_\varepsilon}\right) > 1-\varepsilon$, for all $T > 0$. This gives tightness of the sequence and then the existence of a converging subsequence that we keep denoting by $(Q_T)_{T > 0}$ in the following. \\

\textit{Step 2}. Let us denote by $Q$ the limit of this converging subsequence, and show that $Q$ is an invariant measure. Let $T > 0$, and $F \in C_{b}^{2}(\mathcal{P}(\mathbb{R}^{d}), \R)$, we obtain 
\begin{equation}\nonumber
	\begin{aligned}[t]
		\langle Q_{T}, \mathcal{M} F\rangle = & \int_{\mathcal{P}(\mathbb{R}^d)} \Big( \int_{\mathbb{R}^{d}}[D_{m} F(m, x) \cdot b(x, m)+\frac{\sigma^2+ \sigma_0^2}{2} \nabla\cdot (D_{m} F(m,x))] m(\mathrm{d}x) \\
		& +\frac{\sigma_0^2}{2} \int_{\mathbb{R}^{2d}} \mathrm{Tr}[D_{mm}^{2} F\left(m, x, y\right)] m(\mathrm{d}x) m(\mathrm{d}y)\Big) Q_{T}(\mathrm{d}m) \\
		& = {T}^{-1}\int_0^{T} \int_{\mathcal{P}(\mathbb{R}^d)} \Big( \int_{\mathbb{R}^{d}}[D_{m} F(m, x) \cdot b(x, m)+\frac{\sigma^2+ \sigma_0^2}{2} \nabla\cdot (D_{m} F(m,x))] m(\mathrm{d}x) \\
		& +\frac{\sigma_0^2}{2} \int_{\mathbb{R}^{2d}} \mathrm{Tr}[D_{mm}^{2} F(m, x, y)] m(dx) m(dy)\Big) P_t(\mathrm{d}m) \\
		& = {T}^{-1} \langle P_{T}-P_{0}, F\rangle,
	\end{aligned}
\end{equation} 
{where at the last line, we use the chain rule proved in \cite[Vol. I, Chap. 5]{carmona2018probabilistic} and \cite[Vol. II, Chap. 4]{carmona2018probabilistic}}.
Hence, for all $F \in C_{b}^{2}(\mathcal{P}(\mathbb{R}^{d}))$, $\left\langle Q, \mathcal{M} F\right\rangle = 0$. This ensures that $Q$ is an invariant measure. \\

\textit{Step 3.} Finally we move on to the moment estimate. We know that there exists a subsequence of $Q_T$ which converges weakly to $\widebar P$.
Moreover,
	\begin{equation}\nonumber
	\int_{\mathcal{P}\left( \R^d\right)}\int_{\R^d} | x |^2 m(\mathrm{d}x) Q_T(\mathrm{d}m) = \frac{1}{T}\int_0^T\int_{\mathcal{P}\left(\R^d\right)} \int_{\R^d} | x |^2 m(\mathrm{d}x)P_t(\mathrm{d}m) \mathrm{dt},
\end{equation}
and \begin{equation}\nonumber
	\sup_{T > 0}  \frac{1}{T}\int_0^T\int_{\mathcal{P}\left(\R^d\right)} \int_{\R^d} | x |^2 m(\mathrm{d}x)P_t(\mathrm{d}m) \mathrm{dt} < +\infty,
\end{equation}
 as $\sup_{t>0}\E\left[ |X_t|^2\right] < +\infty$. Moreover,  the function $m \in \mathcal{P}_2\left(\R^d\right) \mapsto \int_{\R^d}|x|^2 m(\mathrm{d}x)$ is lower semi-continuous, and then,
 \begin{equation}\nonumber
 	\int_{\mathcal{P}\left(\R^d\right)} \int_{\R^d} | x|^2 m(\mathrm{d}x) \widebar P\left( \mathrm{d}m\right) \leq \liminf _{T \to +\infty} \frac{1}{T}\int_0^T\int_{\mathcal{P}\left(\R^d\right)} \int_{\R^d} | x |^2 m(\mathrm{d}x)P_t(\mathrm{d}m) \mathrm{dt} < +\infty.
 \end{equation}
\end{proof}

		%The case of a uniformly convex confinement potential
		\section{The case of a uniformly convex confinement potential}\label{se:se3}

In this section, we show that the process $(m_t)_{t\geq 0}$, driven by equation \eqref{eq:main}, has a unique invariant measure under strong convexity assumptions on the confinement potential. Moreover, our method allows us to find exponential rates of convergence toward the invariant measure for a specific set of initial conditions. We again consider $P_0 \in \mathcal{P}_2(\mathcal{P}_2(\R^d))$ and a random variable $X_0$ such that $\mathcal{L}\left( X_0\right) = P_0$, following Definition \ref{D1}. Next, we study the stochastic process $(X_t)$ driven by equation \eqref{LP} with the initial condition $X_0$. In this part of the paper, we consider strict convexity assumptions on the confinement potential $V$. To be specific, we adopt the following assumptions throughout this section:
{
\begin{assumption}\label{A3} $V : \R^d \to \R$ is twice differentiable and satisfies
	\begin{itemize}
		\item[{\rm (A3.1)}] $V$ is uniformly convex, more precisely, there exists $\beta > 0$ such that:
		\begin{equation}\nonumber
			\nabla^2 V \geq \beta \: \mathrm{\mathrm{Id}}.
		\end{equation}
		\item[{\rm (A3.2)}] $\nabla V$ is Lipschitz continuous.
	\end{itemize}
\end{assumption}

\begin{assumption}\label{A4} $W : \R^d \to \R$ is twice differentiable and satisfies
	\begin{itemize}
		\item[{\rm (A4.1)}] $W$ is symmetric and convex;
		\item[{\rm (A4.2)}] $\nabla W$ is globally Lipschitz continuous, with Lipschitz constant $L_W$.
	\end{itemize}
\end{assumption}
}

 This section starts with a key result that states uniform in-time propagation of chaos. This result will then help us establish the uniqueness of the invariant measure and to give a rate of convergence to equilibrium. Thanks to the previous section, under Assumptions \ref{A3} \& \ref{A4}, a process $X$ driven by equation \eqref{LP} admits at least an invariant measure in $\mathcal{P}_2(\mathcal{P}_2(\R^d))$.
 
	%Uniform in time propagation of chaos
	\subsection{Uniform in time propagation of chaos}
In the case without common noise, the uniqueness of the invariant measure for the process $X$ has already been established see $e.g$ \cite{cattiaux2008probabilistic}, \cite{malrieu2001logarithmic}, and \cite{bolley2013uniform}.
In this section, our aim is to adapt this result to the case with common noise and obtain the uniqueness of the invariant measure for the probability measure-valued stochastic process $(m_t)_{t\geq 0}$, and exponentially fast convergence to the equilibria. 
\begin{theorem}\label{T2}
	Let us consider $P,Q \in \mathcal{P}_2(\mathcal{P}_2(\R^d))$, and two particle systems $\bs{X} = (X^1, \dots, X^N)$ and $\bs{X}^N = (X^{1,N}, \dots, X^{N,N})$ with dynamics,
\begin{equation}\label{eq:IPS1}
	\mathrm{d} X^i_t = -\nabla V(  X^i_t) \mathrm{d}t - \nabla W\ast m_t( X^i_t) \mathrm{d}t + \sigma \mathrm{d}B^i_t + \sigma_0\mathrm{d}B^0_t , \quad \forall i \in \left\{ 1, \dots, N\right\},
\end{equation}
and 
\begin{equation}\label{eq:IPS2}
			\mathrm{d} X^{i,N}_t = -\nabla V( X^{i,N}_t) \mathrm{d}t - \frac{1}{N}\sum_{j=1}^N\nabla W(X^{i,N}_t - X^{j,N}_t) \mathrm{d}t + \sigma \mathrm{d}B^i_t + \sigma_0\mathrm{d}B^0_t , \quad \forall i \in \left\{ 1, \dots, N\right\}. 
\end{equation}
where $(X^{i}_0)_{i \in \left\{ 1, \dots, N\right\}}$ are i.i.d such that $\mathcal{L}( \widebar X^{1}_0) = P$ in the sense of Definition \ref{D1}, and where the same holds for the second system with $\mathcal{L}( X^{1,N}_0) = Q$. Then, under Assumption \ref{A3}, there exists a constant $C > 0$ depending only on the dimension $d$ and the probability measures $P$ and $Q$, such that
\begin{equation}\nonumber
	\E\Big[ \mathrm{d}_2^{\R^d}( m^{P,N}_{\bs{X}_t}, m^{Q,N}_{\bs{X}^N_t}) \Big]\leq C\Big(e^{-\beta t} + \frac{1}{2\beta\sqrt{N-1}}\Big),
\end{equation}
where $m^{P,N}_{ \bs{X}_t} = \frac{1}{N}\sum_{i=1}^N \delta_{ X^i_t} $ and $m^{Q,N}_{\bs{X}^N_t} =  \frac{1}{N}\sum_{i=1}^N \delta_{X^{i,N}_t}$.

\end{theorem}

The proof of this result closely follows the approach presented in \cite{benachour1998nonlinear} and \cite{malrieu2001logarithmic}, which shows the propagation of chaos for particle systems in cases without common noise. However, in our situation, we need to be careful with the interaction term and its dependency on the common noise. 

{
\begin{remark}
The result of Theorem \ref{T2} remains valid even without requiring that $\nabla V$ is Lipschitz-continuous, up to some higher moment control. The primary challenge lies in ensuring the existence and uniqueness of a strong solution to \eqref{eq:IPS1} and \eqref{eq:IPS2}. This can be achieved using similar arguments to those in \cite{benachour1998nonlinear}, adapted to account for the presence of common noise, provided the initial condition has a moment of order 4. Here, we leverage the Propagation of Chaos stated in Theorem \ref{T2} to establish the uniqueness of the invariant measure in $\mathcal{P}_2(\mathcal{P}_2(\R^d))$. While it is possible to relax the Lipschitz continuity assumption to achieve uniqueness in $\mathcal{P}_2(\mathcal{P}_4(\R^d))$, this is not the purpose of the current framework.
\end{remark}
}

\begin{proof} We only sketch the proof, as it follows closely the proof of \cite[Thm. 3.3]{malrieu2001logarithmic}. Let $i \in \{ 1, \dots, N\}$, using Itô formula, we get
%\begin{align*}
%	\frac{\mathrm{d}}{\mathrm{d}t}| X^{i,N}_t -  X^{i}_t|^2 \leq -2 \beta |X^{i,N}_t -  X^{i}_t|^2-\frac{2}{N}\sum_{j=1}^N ( X^{i,N}_t -  X^{i}_t ) \cdot (\nabla W( X^{i,N}_t - X^{j,N}_t) - \nabla W\ast m_t( X^{i}_t))
%\end{align*}
\begin{align*}
	 \frac{\mathrm{d}}{\mathrm{d}t}| X^{i,N}_t -  X^{i}_t|^2 & = -2 (X^{i,N}_t -  X^{i}_t)\cdot(\nabla V(X^{i,N}_t) - \nabla V(X^{i}_t)) \\
	 & -\frac{2}{N}\sum_{j=1}^N ( X^{i,N}_t -  X^{i}_t ) \cdot (\nabla W( X^{i,N}_t - X^{j,N}_t) - \nabla W\ast m_t( X^{i}_t)).
\end{align*}
To control the second term, we make the following decomposition:
\begin{align*}
	&  -2 ( X^{i,N}_t -  X^{i}_t )\Big( N^{-1}\sum_{j=1}^N\nabla W( X^{i,N}_t - X^{j,N}_t) - \nabla W\ast m_t( X^{i}_t)\Big) \\
	 &= -2 ( X^{i,N}_t -  X^{i}_t)\Big( N^{-1}\sum_{j=1}^N\nabla W( X^{i,N}_t - X^{j,N}_t) - N^{-1}\sum_{j=1}^N\nabla W(  X^{i}_t -  X^{j}_t)\Big)\\
	  & \hspace{50pt}  - 2 ( X^{i,N}_t -  X^{i}_t)\Big(N^{-1}\sum_{j=1}^N\nabla W( X^{i}_t -  X^{j}_t) - \nabla W\ast m_t( X^{i}_t)\Big) \\
	&= \Xi^{i,N}_t + \Upsilon^{i,N}_t,
\end{align*}
where 
{
\begin{equation*}
\Xi^{i,N}_t = -2 ( X^{i,N}_t -  X^{i}_t)\Big( N^{-1}\sum_{j=1}^N\nabla W( X^{i,N}_t - X^{j,N}_t) - N^{-1}\sum_{j=1}^N\nabla W(  X^{i}_t -  X^{j}_t)\Big)
\end{equation*}
and 
\begin{equation*}
\Upsilon^{i,N}_t = - 2 ( X^{i,N}_t -  X^{i}_t)\Big(N^{-1}\sum_{j=1}^N\nabla W( X^{i}_t -  X^{j}_t) - \nabla W\ast m_t( X^{i}_t)\Big). 
\end{equation*}

Let us now write that for any $i \in \{1, \dots, N\}$, 
\begin{equation*}
	 \Xi^{i,N}_t = -\frac{2}{N}\sum_{j=1}^N \xi^{i,j,N}_t,
\end{equation*}
with 
\begin{equation*}
	\xi^{i,j,N}_t = ( X^{i,N}_t -  X^{i}_t)\Big(\nabla W( X^{i,N}_t - X^{j,N}_t) - N^{-1}\sum_{j=1}^N\nabla W(  X^{i}_t -  X^{j}_t)\Big).
\end{equation*}

Summing the first term over $i$, we obtain
\begin{align*}
	\frac{1}{N} \sum_{i=1}^N \Xi^{i,N}_t & = -N^{-2}\sum_{i,j=1}^N ( \xi^{i,j,N}_t + \xi^{j,i,N}_t).
\end{align*}
Moreover, for any, $i, j \in \{ 1, \dots, N\}$,
\begin{equation*}
\xi^{i,j,N}_t + \xi^{j,i,N}_t = ( X^{i,N}_t - X^{j,N}_t -  X^{i}_t +  X^j_t )( \nabla W( X^{i}_t -  X^{j}_t) - \nabla W(  X^{i}_t -  X^{j}_t)).
\end{equation*} 
Then, using the convexity assumption on $W$, we obtain
\begin{equation*}
	\frac{1}{N} \sum_{i=1}^N \Xi^{i,N}_t \leq 0. 
\end{equation*}
}
{
The control of $\Upsilon^{i,N}_t$ is a bit more intricate. First, let us note that we can write:
\begin{align*}
\Upsilon_t^{i, N}=2 \frac{1-N}{N}\left(X_t^{i, N}-X_t^i\right)\Big[\Big(\frac{1}{N-1} \sum_{j \neq i} \nabla W\left(X_t^i-X_t^j\right)-& \nabla W * m_t\left(X_t^i\right)\Big)\\
& +\frac{1}{N} \nabla W * m_t\left(X_t^i\right)\Big].
\end{align*}
This allows us to say that
\begin{equation}
	\E[|\Upsilon^{i,N}_t|] \leq \barroman{I} + \barroman{II},
\end{equation}
where 
\begin{align*}
	& \barroman{I} = 2\frac{N-1}{N}\E\left[ \left| X^{i,N}_t - X^i_t\right|^2\right]^{1/2}\E\left[ \left| \frac{1}{N-1} \sum_{j \neq i} \nabla W\left(X_t^i-X_t^j\right)-\nabla W * m_t\left(X_t^i\right)\right|^2\right]^{1/2}; \\
	& \barroman{II} = 2\frac{N-1}{N}\E\left[ \left| X^{i,N}_t - X^i_t\right|^2\right]^{1/2}\E\left[\left| \nabla W \ast m_t(X^i_t)\right|^2\right]^{1/2}.
\end{align*}

Let us begin with the second term $\barroman{II}$. First, we observe that for any $t > 0$, 
\begin{align*}
	\E\left[|\nabla W\ast m_t(X^i_t)|^2\right] & = \E\left[ \left|\int_{\R^d} \nabla W(X^i_t - y)m_t(\mathrm{d}y) \right|^2\right] \\
	& \leq L_W^2 \E\left[ \left| \int_{\R^d} |X^i_t|m_t(\mathrm{d}y) + \int_{\R^d} |y|m_t(\mathrm{d}y)\right|^2\right] \\
	& \leq L_W^2 \E\left[ \left| |X^i_t|  + \E[|X^i_t| \: |\mathcal{F}^0_t]\right|^2\right] \\
	& \leq 2L_W^2 \E\left[ |X^i_t|^2\right].
\end{align*} 
This gives
\begin{equation}\label{eq:II}
	\barroman{II} \leq \frac{2\sqrt{2L_W}(N-1)}{N^2}\E\left[ |X^i_t|^2\right]^{1/2}\E\left[ \left| X^{i,N}_t - X^i_t\right|^2\right]^{1/2}.
\end{equation}
For the term $\barroman{I}$, we can write for any $t > 0$
\begin{align*}
	&\E[ | \nabla W\ast m_t(X^i_t) - \frac{1}{N-1}\sum_{j=1}^N \nabla W(  X^i_t -  X^j_t)|^2] \\
    & = \E\Big[ \E\Big[ \Big| \nabla W\ast m_t(X^i_t) - \frac{1}{N-1}\sum_{j=1}^N \nabla W( X^i_t -  X^j_t)\Big|^2\Big|  X^i_t, \mathcal{F}^0_t\Big]\Big] \\
    & = \E\Big[ \mathrm{Var}\Big[  \frac{1}{N-1}\sum_{j=1}^N \nabla W(  X^i_t -  X^j_t)\Big|  X^i_t, \mathcal{F}^0_t\Big]\Big] \\
	& \leq \frac{1}{N-1}\E[\E[ | \nabla W(  X^i_t - X^j_t)|^2|  X^i_t, \mathcal{F}^0_t] ], \qquad \text{for some $j \neq i$} \\
	& \leq \frac{L_W^2}{N-1}\E[\E[ | X^i_t -  X^j_t|^{2}|  X^i_t, \mathcal{F}^0_t]] \\
	& \leq \frac{2 L_W^2}{N-1} \E[ |  X^1_t|^{2}],
\end{align*}

where the third line comes from the fact that $\E[ \nabla W(  X^i_t -  X^j_t) |  X^i_t, \mathcal{F}^0_t] = \nabla W\ast m_t(  X^i_t)$, and the fourth from the fact that the $(X^j_t)_{j\neq i}$ are  conditionally independent w.r.t. $(X^i_t, \mathcal{F}^0_t)$.  This gives
\begin{equation}\label{eq:I}
	\barroman{I} \leq \frac{4 L_W}{\sqrt{N}} \E[ |  X^1_t|^{2}]^{1/2}\E\left[ \left| X^{i,N}_t - X^i_t\right|^2\right]^{1/2}.
\end{equation}

Combining \eqref{eq:II} and \eqref{eq:I}, we obtain that there exists a constant $C_1 > 0$ independent of $t$ and $N$ such that
\begin{equation*}
	\E[|\Upsilon^{i,N}_t|] \leq \frac{C_1}{\sqrt{N-1}}\E\left[\left|X^1_t\right|^2\right]^{1/2}\E\left[\left|X^i_t - X^{i,N}_t\right|^2\right]^{1/2}.
\end{equation*}
Then, thanks to Lemma \ref{le:L2}, we obtain that there exists a constant $C > 0$, such that
\begin{align*}
	\E[|\Upsilon^{i,N}_t|]  \leq \frac{C}{\sqrt{N-1}}\E[| X^{i,N}_t - X^i_t|^2]^{1/2}.
\end{align*}
Finally, we get that
\begin{align*}
	\frac{1}{N}\sum_{i=1}^N\frac{\mathrm{d}}{\mathrm{d}t} \E[ |  X^{i}_t - X^{i,N}_t|^2] \leq & -\frac{2\beta}{N}\sum_{i=1}^N \E[ |  X^{i}_t - X^{i,N}_t|^2] \\
	& \hspace{20pt}+ \frac{C}{\sqrt{N-1} }\Big( N^{-1}\sum_{i=1}^N\E[| X^{i,N}_t -  X^{i}_t|^2]\Big)^{1/2},
\end{align*}
for some constant $C > 0$.
Let us now denote $v_N(t) = N^{-1}\sum_{i=1}^N\E[ |  X^{i}_t - X^{i,N}_t|^2]$, then we have
\begin{equation}\nonumber
	v'_N(t) \leq -2\beta v_N(t) + \frac{C}{\sqrt{N-1}} v_N(t)^{1/2}. 
\end{equation}
%-2 \left( X^{i,N}_t - \widebar X^{i}_t\right) \left(\left( N-1\right)^{-1}\sum_{i=1}^N\nabla W\left( X^{i,N}_t - X^{j,N}_t\right) - \nabla W\left(\widebar X^{i}_t\right)\right)
This gives, using Gronwall Lemma:
\begin{equation}\label{eq:PCUT}
	v_N(t)^{1/2}\leq e^{-\beta t}v_N(0)^{1/2} + \frac{C}{2\beta\sqrt{N-1}}.
\end{equation}
Moreover, $v_N(0) = N^{-1}\sum_{i=1}^N\E[ |  X^{i}_0 - X^{i,N}_0|^2] \leq \E[ |X^{i}_0|^2 + |X^{i,N}_0|^2]$ is bounded uniformly in $N$,
and  as $\E[ \mathrm{d}_2^{\R^d}( m^N_{\bs{X}_t}, m^N_{\bs{X}^N_t}) ]\leq v_N(t)^{1/2},$ we conclude the proof of Theorem \ref{T2}. }
\end{proof}
	
	%Uniqueness of the invariant measure
	\subsection{Uniqueness of the invariant measure}
	
The main consequence of the previous result is the uniqueness of the invariant measure for the process $(m_t)_{t\geq 0}$ driven by \eqref{eq:main}. As recalled at the beginning of the section, we have already shown that under Assumptions \ref{A1} and \ref{A2}, there exists an invariant measure $\widebar P$. From Theorem \ref{T2}, we get the following Corollary
\begin{corollary}\label{C1}
	Under Assumptions \ref{A3}, the stochastic process $(m_t)_{t\geq 0}$ admits a unique invariant measure $\widebar P \in \mathcal{P}_2\left(\mathcal{P}\left( \R^d\right)\right)$. Moreover, for each $P_0 \in \mathcal{P}_2(\mathcal{P}_2(\R^d))$, there is an exponential convergence to the invariant measure:
	\begin{equation}\nonumber
	\mathrm{d}_2^{\mathcal{P}(\R^d)}( P_t, \widebar P) \leq e^{-\beta t}\mathrm{d}_2^{\mathcal{P}(\R^d)}( P_0, \widebar P)^2.
	\end{equation}
\end{corollary}

This implies the uniqueness of the invariant measure and the convergence to this equilibria for a large class of initial conditions $P_0$. 
To prove the previous result, we begin with a technical Lemma:
\begin{lemma}\label{MS}
    Let $m$ and $\rho$ be two probability measure-valued random variables which are $\mathcal{F}^0_0$-measurable. Then, there exists a random variable $\xi$ defined on the space $(\Omega, \mathcal{F}, \mathbb{P})$ and with value in $\mathcal{P}(\R^d)$, such that almost surely:
    \begin{equation}\nonumber
        \xi \in \argmin_{\pi \in \Pi\left( m, \rho\right)} \int_{\R^d}|x-y|^2 \pi(\mathrm{d}x, \mathrm{d}y).
    \end{equation}
\end{lemma}
The proof of this Lemma is postponed to Appendix \ref{appen:SM} and relies on measurability arguments for set-valued functions issued from  \cite{stroock1997multidimensional}.

\begin{proof}[Proof of Corollary \ref{C1}]
	The proof of this result relies on the result and the proof of Theorem \ref{T2}, but the important difference is the choice of the initial conditions. {Let us first consider an invariant measure $\widebar P$, which existence is ensured by Proposition \ref{T1}}. Then, for $P_0 \in \mathcal{P}_2(\mathcal{P}_2(\R^d))$, we pick $\Gamma \in \Pi\left( P_0, \widebar P\right)$ which is not empty.  Let us consider a couple of probability measure valued random variables $(m_0, \widebar m_0)$, and such that $\mathcal{L}\left( (m_0, \widebar m_0)\right) = \Gamma$. It means that $\mathcal{L}(m_0) = P_0$ and $\mathcal{L}(\bar m_0) = \bar P_0$. Thanks to Lemma \ref{MS}, we know that there exists $\xi$ random variable such that almost surely,
\begin{align*}
	\xi \in \argmin_{\pi \in \Pi\left( m_0, \widebar m_0\right)} \int_{\R^d}|x-y|^2 \pi(\mathrm{d}x, \mathrm{d}y).
\end{align*}
We consider once again the particle system:
	\begin{equation}\nonumber
	\mathrm{d} X^i_t = -\nabla(  X^i_t)\mathrm{d}t - \nabla W\ast m_t(  X^i_t) \mathrm{d}t + \sigma \mathrm{d}B^i_t + \sigma_0\mathrm{d}B^0_t , \quad \forall i \in \{ 1, \dots, N\},
\end{equation}
and 
\begin{equation}\nonumber
			\mathrm{d} X^{i,N}_t = -\nabla ( X^{i,N}_t) \mathrm{d}t - \frac{1}{N}\sum_{j=1}^N\nabla W(X^{i,N}_t - X^{j,N}_t) \mathrm{d}t + \sigma \mathrm{d}B^i_t + \sigma_0\mathrm{d}B^0_t , \quad \forall i \in \{ 1, \dots, N\},
\end{equation}
where the $( X^i_0, X^{i,N}_0)$ for $i \in \{ 1, \dots, N\}$ are independent and such that $\mathcal{L}( ( X^i_0, X^{i,N}_0) | \mathcal{F}^0_0) = \xi$. 
	From Theorem \ref{T2}, and more precisely Equation \eqref{eq:PCUT}, we know that there exists a constant $C > 0$, such that:
	\begin{equation}\label{eq:propchaos}
		\E[\mathrm{d}_2^{\R^d}( m^N_{\bs{X}_t}, m^N_{\bs{X}^N_t})^2 ] \leq e^{-\beta t}\E[ | X^1_0 - X^{1,N}_0|^2] + CN^{-1/2}.
	\end{equation}
Moreover, with the particular choice we made for the initial conditions, we get that
\begin{equation}\nonumber
\begin{aligned}[t]
	\E[ | X^1_0 - X^{1,N}_0|^2] & =  \E[ \E[ | X^1_0 - X^{1,N}_0|^2|\mathcal{F}^0_0]]\\
	& = \E\left[ \int_{\R^d} |x-y|^2 \xi( \mathrm{d}x, \mathrm{d}y)\right] \\
	& = \int_{\mathcal{P}(\R^d)} \mathrm{d}_2^{\R^d}( \mu, \nu) \Gamma(\mathrm{d}\mu, \mathrm{d}\nu).
\end{aligned}
\end{equation} 
{Taking the infimum over all the transport plans $\Gamma$ between $P_0$ and $\widebar P$ in \eqref{eq:propchaos}, we get that 
\begin{equation}\label{eq:eq1}
	\E\Big[\mathrm{d}_2^{\R^d}( m^N_{\bs{X}_t}, m^N_{\bs{X}^N_t})^2 \Big] \leq e^{-\beta t} \mathrm{d}_2^{\mathcal{P}( \R^d)}( P_0, \widebar P ) + \frac{C}{\sqrt N}.
\end{equation}
Taking now the limit $N \to +\infty$, we obtain
\begin{equation*}
	\E\Big[\mathrm{d}_2^{\R^d}( \widebar m_t, m_t)^2 \Big] \leq  e^{-\beta t} \mathrm{d}_2^{\mathcal{P}( \R^d)}( P_0, \widebar P),
\end{equation*}
where $(\widebar m_t)$ is the stochastic flow of measure driven by \eqref{eq:main} with initial condition $\widebar P$ and $(m_t)_{t\geq 0}$ has the same dynamic with initial condition $P_0$. 
Finally, as 
\begin{equation}\label{eq:eq2}
	\mathrm{d}_2^{\mathcal{P}(\R^d)}(P_t, \widebar P) \leq \E\Big[\mathrm{d}_2^{\R^d}( \widebar m_t, m_t)^2 \Big], 
\end{equation}
we obtain
\begin{equation}\nonumber
	\mathrm{d}_2^{\mathcal{P}( \R^d)}( P_t, \widebar P) \leq e^{-\beta t}\mathrm{d}_2^{\mathcal{P}(\R^d)}( P_0, \widebar P).
\end{equation}
}
\end{proof}

The previous result proves that common noise does not ruin the usual convergence results. It makes sense because the common noise is finite-dimensional and mainly acts like a drift term in Equation \eqref{eq:main} and should not make the usual convergence results go haywire.
	
	%An interesting example
	\subsection{Discussion and example}

In this section, we consider an Ornstein-Uhlenbeck process with common noise, that is taking $V : x \mapsto |x|^2/2$ and $W =0$. Let $X$ be a real valued stochastic process evolving in $\R^d$, driven by the following stochastic differential equation:

\begin{equation}
	\mathrm{d}X_t = -X_t\mathrm{d}t +\sigma \mathrm{d}B_t+\sigma_0\mathrm{d}B^0_t. 
\end{equation}
Then, the process associated process $(m_t)_{t\geq 0}$ is solution of 
\begin{equation}
	\mathrm{d}_t  m_t = \nabla \cdot \left[\frac{\sigma^2 + \sigma_0^2}{2} \nabla m_t + m_t x\right]\mathrm{d}t -\sigma_0 \nabla m_t \cdot \mathrm{d}B^0_t,
\end{equation}
Moreover, we consider the following initial condition $\mathcal{L}\left( X_0\right) = P_0$ (in the sense of Definition \ref{D1}), for some $P_0 \in \mathcal{P}_2(\mathcal{P}_2(\R^d))$.  
In this particular setting, we can explicitly describe the invariant measure.\begin{proposition}\label{P4}
The unique invariant measure $\widebar P$ of the process on the space $\mathcal{P}( \mathbb{R}^d)$ is the image, by the function $x \in \mathbb{R}^d \mapsto \mathcal{N}_d( -x, \sigma^2\mathrm{Id}) \in \mathcal{P}( \mathbb{R}^d)$ of the measure $\gamma_{\sigma_0}$; where $\gamma_{\sigma_0}(\mathrm{d}x) = (2\sigma_0^2)^{-1/2}e^{-|x|^2/2\sigma_0^2}\mathrm{d}x $, where for all $\mu \in \R^d, \Sigma \in \mathcal{M}_d(\R)$, $\mathcal{N}_d\left( \mu, \Sigma\right)$ denotes a gaussian distribution in dimension $d$ centered in $\mu$ and with variance-covariance matrix $\Sigma$.  
\end{proposition}

The fact that in this case, we are able to exhibit the invariant measure comes from the linearity of the equation and the linearity of the confinement forces which derives from a quadratic potential, as shown in the proof.  The convergence to the equilibria holds at an exponential thanks to Theorem \ref{T2}. 

\begin{proof}[ Proof of Proposition \ref{P4}]
Let be the process $X^0$ defined dynamic given by: 
\begin{equation}\nonumber
\left\{ \begin{array}{ll}
dX^0_t = -X^0_t\mathrm{d}t -\sigma_0\mathrm{d}B^0_t \\
\mathcal{L}\left( X^0_0\right) = \mathcal{N}\left( 0, \sigma_0^2\right). 
\end{array}\right.
\end{equation}
The initial condition is such that the process $X^0$ is stationary. We now define the process $\left(\widebar m_t\right)_t$ with $\widebar m_0$, such that $\mathcal{L}(\bar m_0) = P$ as initial condition and staying
\begin{equation}\label{uakzyhejnsdf0}
\mathrm{d}_t \widebar m_t = \nabla \cdot \Big[\frac{\sigma^2 + \sigma_0^2}{2} \nabla\widebar m_t + \widebar m_t x\Big]\mathrm{d}t -\sigma_0\nabla\widebar m_t \cdot \mathrm{d}B^0_t. 
\end{equation}
Let us define $\widetilde m_t:= (\mathrm{Id}- \sigma^0 B^0_t)\sharp \widebar m_t$, where $\sharp$ stands for the pushforward operator. Then, applying the Itô-Wentzell formula, we get that $\widetilde m_t$ satisfies
\begin{equation}\label{uakzyhejnsdf}
\left\{ \begin{array}{ll}
\partial_t \widetilde m_t = \nabla\cdot[\frac{\sigma^2}{2} \nabla\widetilde m_t +\widetilde m_t (x+B^0_t)],  \\
\widetilde m_0 = \widebar m_0.
 \end{array}\right.
\end{equation}
Moreover, we get that $m_t:= \mathcal N(-(X^0_t+\sigma_0B^0_t),\sigma^2 \mathrm{Id})$ satisfies, 
\begin{equation}
\mathrm{d}_t m_t(x)= - \sigma^{-2}(x+X^0_t+\sigma^0B^0_t) m_t(x)\:\mathrm{d}(X^0_t+\sigma_0B^0_t) =   \sigma^{-2}(x+X^0_t+\sigma_0B^0_t) \cdot X^0_t m_t(x)\: \mathrm{d}t.
\end{equation}
Then, as 
{
\begin{align*}
	\frac{1}{2} \sigma_0^2 \Delta m_t+\operatorname{div}\left(m_t\left(x+B_t^0\right)\right)=-\sigma^{-2} X_t^0 \nabla m_t=\sigma^{-2} X_t^0 \cdot\left(x+X_t^0+\sigma_0 B_t^0\right) m_t(x),
\end{align*}
}
it shows that $(m_t)_{t\geq 0}$ satisfies \eqref{uakzyhejnsdf} with $m_0= \widebar m_0$ and then $( (\mathrm{Id}+\sigma^0B^0_t)\sharp m_t)_t$ is solution of \eqref{uakzyhejnsdf0} with the same initial condition. Finally, $$\widebar m_t(\mathrm{d}x)=  (\mathrm{Id}+\sigma^0B^0_t)\sharp m_t(\mathrm{d}x) = c\exp\{-|x+X^0_t|^2/2\sigma^2\}\mathrm{d}x. $$ 
As $(X^0_t)_t$ is stationary in $\mathbb{R}^d$, $\widebar m_t$ is also in $\mathcal{P}_2(\mathbb{R}^d)$, which shows the first part of the Proposition. The uniqueness of this invariant measure is then given by Corollary \ref{C1}. 
\end{proof}

%%%%%%%%%%%%%%%%%
\section{Non-convex potential without idiosyncratic noise}\label{sec:se4}

In this section, we consider the case of a non-convex potential $V$ in the particular setting $\sigma = 0$. It turns out that in some specific situations, we can use the presence of interaction to achieve an exponential rate of convergence to the unique invariant measure. In this section, let us consider the particular case of a quadratic interaction potential $W$. More precisely, we consider that there exists $\alpha > 0$ such that for any $x \in \R^d$, $W(x) = \alpha |x|^2/2$. \\

 More precisely, we consider $m$ with dynamic given by 
\begin{equation}\label{eq:SFP:sigma0}
\mathrm{d}_t  m_t = \nabla \cdot \Big(\frac{\sigma_0^2}{2} \nabla m_t + m_t (\nabla V + \alpha(\cdot - \mu_1(m_t)))\Big)\mathrm{d}t  -\sigma_0 \nabla m_t \cdot \mathrm{d}B^0_t,
\end{equation} 
where, for any $m \in \mathcal{P}_1(\R^d)$, 
\begin{equation*}
	\mu_1(m) = \int_{\R^d} x \: m(\mathrm{d}x) \in \R^d. 
\end{equation*}

For the potential $V$, we will consider the following assumptions:
{
\begin{assumption}\label{A5}
	The drift $V : \R^d \to \R$ is twice differentiable in $\R^d$. Moreover, 
	\begin{itemize}
	\item[\rm (A5.1)]$V$ is confining in the sense that there exists a function $\kappa : [0; +\infty) \to \R$ 
		 {such that} 
    \begin{equation}\nonumber  
        \left( \nabla V(x) - \nabla V(y)\right)\cdot\left(x-y\right) \geq \kappa(| x-y|) |x-y|^2.
    \end{equation}
		{with}
		\begin{equation*}
		\lim\sup_{r \to +\infty}\kappa(r) > 0 \quad 
		\textrm{\rm and} \quad 
			\int_0^1 r \kappa(r)^{-} \mathrm{d} r<\infty, 
		\end{equation*}
		where $\kappa^{-} = \max(0, -\kappa)$. 
	\item[\rm (A.5.2)] $\nabla V$ is Lipschitz continuous with constant $L_V > 0$. 
\end{itemize}
\end{assumption}
}

{
\subsubsection*{About the Assumptions}

The assumptions in this section are different from those in the earlier parts. Here, we explain these differences and the reasons behind them:

\begin{itemize}

	\item \textbf{The Assumptions on $V$}. Once again, we assume that $\nabla V$ is Lipschitz-continuous. However, this assumption is no longer merely a technical requirement to ensure strong uniqueness of the solutions of Equation \eqref{LP} but carries significant physical relevance in this context. From Equation \eqref{eq:SFP:sigma0}, it is clear that the first-order moment $(\mu_1(m_t))_{t \geq 0}$ plays a crucial role in the dynamics. The regularity assumption on $\nabla V$ enables us to demonstrate that the introduction of common noise sufficiently randomizes $(\mu_1(m_t))_{t \geq 0}$, ensuring its ergodicity---even for an arbitrarily small level of common noise $\sigma_0 > 0$. This phenomenon lies at the heart of uniqueness recovery.

    \item \textbf{The Structure of the Interaction}. In this Section, we only consider \textit{linear interactions}. This choice is important to ensure the uniqueness of the invariant measure. As previously mentioned, to achieve this, we use the fact that the common noise randomizes the mean enough (first moment) of the measures $(m_t)_{t\geq 0}$. Then, the linear structure of the interaction also has another benefit: it makes the variance of $m_t$ decrease over a long period. This decrease helps the system settle into a unique state. With a more general structure (non-quadratic interaction force), we cannot say that the common noise helps recover uniqueness even in the case $\sigma = 0$. This is due, as previously mentioned, to the difficulty that the common noise here is only finite-dimensional. 
\end{itemize}
\color{black}
\subsection{Existence of an invariant measure}

This section begins with a discussion on the existence of an invariant measure, which in this case can be described explicitly. This comes from the absence of idiosyncratic noise within our system coupled with the strong convexity of the interaction, which allows the measure-valued process driven by Equation~\eqref{eq:SFP:sigma0}, to get close to its first moment on a large time scale. Then, we see that the variance within the stochastic flow of measure decays to zero. It is then natural that the invariant measure finds its support in Dirac masses, as detailed in the following Proposition. 
\begin{proposition}\label{prop:P9}
	Under Assumptions \ref{A5}, {there exists at least one invariant measure $\widebar P$} in the sense of Definition \ref{D1}, and this measure is supported by Dirac masses. More precisely,  
{
$$
\widebar P(\mathrm{d}m)= \int_{\mathbb{R}^d} \delta_{\delta_a}(\mathrm{d}m) m_0(\mathrm{d}a)
$$
}
where $m_0$ is the probability measure on $\mathbb{R}^d$ solution of $m_0$ admits a density proportional to $\exp(-2/\sigma_0^2 V)$.
\end{proposition}

\begin{proof}[Proof of Proposition \ref{prop:P9}]
	If we define $$\widebar P(\mathrm{d}m)= \int_{\mathbb{R}^d} \delta_{\delta_a}(\mathrm{d}m) m_0(\mathrm{d}a), $$ then for all twice differentiable function in the sense of Lions derivatives $F\in\mathcal{C}^2(\mathcal{P}_2(\R^d), \R)$
\begin{equation}
\begin{aligned}[t]
I(F) &:= \int_{\mathcal{P}_2(\mathbb{R}^d)} \left[ \int_{\mathbb{R}^d} \left( D_mF(m,x)\cdot \left(- \nabla V(x) - \nabla W\ast m(x) \right)+\frac{\sigma_0^2}{2} \nabla\cdot D_mF(m,x) \right)) m(\mathrm{d}x) \right. \\
& \left. +\frac{\sigma_0^2}{2}\int_{\mathbb{R}^{2d}} \mathrm{Tr}\left[D^2_{mm}F(m,x,y)) \right]m(\mathrm{d}x)m(\mathrm{d}y)\right] \widebar P(\mathrm{d}m)\\
& = \int_{\mathbb{R}^d}  \left[  D_mF(\delta_a,a)\cdot \left( - \nabla V(a) \right) +\frac{\sigma_0^2}{2} \nabla \cdot D_mF(\delta_a,a)  
+\frac{\sigma_0^2}{2} \mathrm{Tr}\left[D^2_{mm}F(\delta_a,a,a)\right] \right]m_0(\mathrm{d}a),
\end{aligned}
\end{equation}
where, if we define $\varphi(a)= F(\delta_a)$, then 
$$
\nabla \varphi(a)= D_mF(\delta_a,a), \qquad \Delta\varphi(a)= \mathrm{Tr}\left[D^2_{mm}F(\delta_a,a,a)+ D^2_{xm} F(\delta_a,a)\right].
$$
Then,
\begin{equation}
I(F)=  \int_{\mathbb{R}^d}  \left[  - \nabla \varphi(a)\cdot \nabla V(a) +\frac{\sigma_0^2}{2} \Delta \varphi(a)  \right]m_0(\mathrm{d}a)=0, 
\end{equation}
as $m_0$ is solution of  $\frac{\sigma_0^2}{2} \Delta m_0 + \nabla\cdot\left(\nabla Vm_0\right)=0$. This concludes the proof. 
\end{proof}

{ We emphasize here that such an identification of an invariant measure only works in the particular case $\sigma = 0$. 
In fact, in this case, the dynamic of the associated process $X$ is only driven by the so-called common noise $B^0$. Then the key insight is that the interaction force drives the process toward its conditional expectation with respect to the common noise. Basically that means that in a large time range $X_t$ is equal in some sense to $\E[X_t|B^0]$. It is natural to expect that, over a large time range, the variance of the process of interest $(m_t)_{t \geq 0}$ converges to $0$, which justifies the form of $\widebar{P}$. However, when $\sigma > 0$, this is no longer true—the variance of $m_t$ does not vanish as $t \to \infty$, making it impossible to identify an invariant measure.
}

\subsection{Uniqueness of the invariant measure and exponential decay}

The previous section shows that for $\sigma=0$, we can explicitly identify an invariant measure, though its uniqueness is uncertain due to the non-convex nature of the confinement potential. This section establishes that, in the absence of idiosyncratic noise, the invariant measure is indeed unique. More precisely, let us state the following result:
\begin{theorem}\label{th:sg0}
	Whenever $\sigma = 0$ and under Assumptions \ref{A3}, for any initial conditions $P_0, Q_0 \in \mathcal{P}_2(\mathcal{P}_2(\R^d))$,  there exists a constant $C$ such that for any $t \geq 0$,
	\begin{align*}
		\mathrm{d}_1^{\mathcal{P}(\R^d)}( P_t, Q_t) \leq C\Big( e^{-\ell\sigma_0^2t} + e^{-( \alpha - 2L_V)t}\Big),
	\end{align*}
	for some positive constant $\ell$. Moreover, $\bar P$ defined in Proposition \ref{prop:P9} is the unique invariant measure in the sense of Defintion \ref{D2}. 
\end{theorem}

This is the main result of this paper, showing that when $\sigma = 0$, there is only one invariant measure, which contrasts with scenarios lacking common noise where multiple invariant solutions exist as we are in a context with non-convex confinement potential. This theorem reveals that introducing finite dimensional noise to the system uniquely determines the invariant measure, a result that is stronger than the previous one in the literature. Earlier studies (see e.g \cite{angeli2023mckean} ) prove that adding cylindrical noise could achieve invariant measure uniqueness, but our results go further by showing that even finite-dimensional noise is enough for this purpose. ~\\

{Proposition \ref{prop:P9} shows that when \( \sigma = 0 \), we are able to identify the unique invariant measure. Moreover, in Theorem \ref{th:sg0}, we observe that the strong interaction structure facilitates exponentially fast convergence to this invariant measure. This result may initially appear surprising, but the explanation lies in how the interaction term influences the convergence rates. The key insight is that this force effectively drives the process toward its conditional mean, simplifying the long-time analysis of the behavior of this conditional mean process, which itself behaves almost like a standard diffusion process. 

The main challenge is to adapt the reflection coupling machinery introduced by Eberle to account for the presence of common noise. This adaptation is necessary to handle the non-convexity and to prove that the conditional mean is an ergodic process.

}

\subsection{Discussion and Heuristics}
{
Let us first mention once again that the difficulty comes from the fact that the common noise is just of finite dimension while 
the state variable, which should be seen as the conditional marginal law of the system given the common noise, lives in 
a space of infinite dimension. In this context, our result holds true if, in addition to standard confining properties, the mean-field interaction term forces the system to be attracted by its conditional expectation. In particular, in this linear interaction setting, we can see that the dynamic of $(m_t)_{t \geq 0}$ is driven by the behavior of $(\mu_1(m_t))_{t\geq 0}$ and that
\begin{equation*}
	\mathrm{d}_t \mu_1(m_t) = -\nabla V(\mu_1(m_t))\mathrm{d}t + \varepsilon^V_t \mathrm{d}t + \sigma_0\mathrm{d}B^0_t,
\end{equation*}
where 
\begin{equation*}
	\varepsilon_t^V = \nabla V(\mu_1(m_t)) - \int_{\R^d} \nabla V(x) m_t(\mathrm{d}x). 
\end{equation*}

Further, using the Lipschitz continuity of $V$, we can show that $\varepsilon^V_\cdot$ decays exponentially fast to $0$ over time. Then, the dynamics of $(\mu_1(m_t))_{t \geq 0}$ resemble those of the diffusion process $S$, constructed on $(\Omega_0, \mathbb{F}^0, \mathbb{P}_0)$ as the solution of
\begin{equation} \label{eq:diffusion}
    \mathrm{d}S_t = -\nabla V(S_t) \, \mathrm{d}t + \sigma_0 \, \mathrm{d}B^0_t, \quad t \geq 0; \quad S_0 = \mathbb{E}_{\mathbb{P}_1}(X_0).
\end{equation}
It is worth emphasizing, even though it may seem obvious, that \eqref{eq:diffusion} is not a McKean–Vlasov equation but a standard diffusion equation. In particular, the long-time analysis of \eqref{eq:diffusion} falls within a much broader literature, as the study of the long-term behavior of diffusion processes has been a significant topic of interest; see, for instance, \cite{ane2000inegalite, bakry2008simple, bakry1985diffusions}, among others. 

To prove that this process is ergodic, we will use coupling-by-reflection techniques similar to those developed by Eberle \cite{eberle2016reflection}. Specifically, we aim to construct a coupling $(X, Y)$, adapted to the presence of common noise, such that the conditional expectations with respect to the common noise converge over time. This is done properly in the following, mainly in the proof of the restoration of uniqueness stated in Theorem \ref{th:sg0}. }

\begin{remark}
The specific form of the invariant measure, as detailed in Proposition \ref{prop:P9}, showcases a scenario unique to quadratic interactions. However, the uniqueness result of the section would hold true even when we tweak the interaction term to something like 
\begin{equation*}
(x,\mu) \mapsto -\nabla W\left( x - \int y \: \mu(\mathrm{d}y)\right),
\end{equation*}
 provided \(W\) is a uniformly convex potential. The critical aspect we are looking at here is not the specific form of the interaction but how it allows us to identify the process with its mean on a large time scale. Then, the uniqueness of the invariant measure holds for more general interaction potential.
\end{remark}

\subsection{Proof of Theorem \ref{th:sg0}} 
To make the proof easier to read, we highlight its outline here. \\

\noindent \textit{Proof Outline.}
\begin{itemize}
    \item \textit{Step 1.} We construct a coupling \((X, Y)\) inspired by the reflection coupling introduced in \cite{lindvall1986coupling}, adapted to the presence of a common noise. 
    \item \textit{Step 2.} We prove that the long-time behavior of the process \((m_t)_{t\geq 0}\) can be understood by studying the long-time behavior of its first moment \((\mu_1(m_t))_{t \geq 0}\).
    \item \textit{Step 3.} We prove, leveraging the coupling to handle non-convexity, that the presence of common noise allows us to establish ergodicity for the first moment \(\mu_1(m_t)\).
    \item \textit{Step 4.} We combine the results of Step 2 and Step 3 to conclude the proof by proving that the recovered ergodicity for the process \((\mu_1(m_t))_{t \geq 0}\) propagates to the whole measure-valued process \((m_t)_{t\geq 0}\).
\end{itemize}

\begin{proof}
	\textit{Step 1:} \textbf{The coupling argument}.\\
	
	 For a given $\delta > 0$, we define $X^{\delta}$ and $Y^\delta$ with common dynamic given by:
\begin{align*}
	\mathrm{d}X^{\delta}_t = -\nabla V(X^{\delta}_t)\mathrm{d}t - \alpha(X^{\delta}_t - \mathbb{E}_{\mathbb{P}^1}[X^\delta_t])  \mathrm{d}t + \sigma_0\{ \pi_{\delta}( E^\delta_t)\mathrm{d}B^0_t + \lambda_\delta(E^\delta_t)\mathrm{d}\tilde B^0_t\}
\end{align*}
and 
\begin{align*}
	\mathrm{d}Y^{\delta}_t = -\nabla V(Y^{\delta}_t)\mathrm{d}t - \alpha(Y^{\delta}_t -\E_{\mathbb{P}^1} Y^{\delta}_t)  \mathrm{d}t+ \sigma_0\{ (\mathrm{Id} - 2e^\delta_t{e^\delta_t}^T)\pi_{\delta}( E^\delta_t)\mathrm{d}B^0_t + \lambda_\delta(E^\delta_t)\mathrm{d}\tilde B^0_t\}. 
\end{align*}
where 
\begin{itemize}
	\item $B^0$ and $\tilde B^0$ are independent Brownian motions adapted to $\mathbb{F}^0$;
	\item For all $t \geq 0$ and $\delta > 0$, we denote 
	\begin{equation*}
		E^\delta_t = \mathbb{E}_{\mathbb{P}_1}[X^\delta_t - Y^\delta_t];
	\end{equation*}
	\item For all $t \geq 0$ and $\delta > 0$, 
	\begin{equation*}
		e^\delta_t = \begin{cases}   E^\delta_t/|E^\delta_t|   \qquad &\text{if } |E^\delta_t| \neq 0, \\
		0 & \text{otherwise}. 
		\end{cases}
	\end{equation*}
	\item We define a non-decreasing and continuous function $\pi$, such that for $x \in \R^N$, 
\begin{equation}\nonumber
	\pi(x) =  \begin{cases}1 & \text { if }|x|\geq 1 \\ 0 & \text { if } |x| \leq 1/2,
	\end{cases}
\end{equation}
and, consider a non negative function $\lambda$ such that $$\pi(x)^2 + \lambda(x)^2 = 1, \quad \forall x \in \R^N.$$
Moreover, we extend $\pi$ on the whole space, with the constraint that this remains a non-decreasing and Lipschitz continuous function. Finally, we define $\pi_\delta: x \mapsto \pi\left( x/\delta\right)$ and $\lambda_\delta: x \mapsto \lambda\left( x/\delta\right)$. 
\end{itemize}

The first thing to notice is that under the Lipschitz continuity assumptions on $\nabla V$, and the integrability properties of the initial conditions, we have the existence and uniqueness of {a} weak solution. Moreover, thanks to the choice of $\pi_\delta$ and $\lambda_\delta$, the pair $(X^\delta, Y^\delta)$ is a coupling of $(Z, \widetilde Z)$ where for $Z$ and $\widetilde Z$ admits the same dynamic,
\begin{equation}
	\mathrm{d}Z_t = -\nabla V(Z_t) \mathrm{d}t - \alpha(Z_t - \E_{\mathbb{P}_1}[Z_t])\mathrm{d}t + \sigma_0\mathrm{d}B^0_t,
\end{equation}
with initial condition $Z_0 = X_0$ and $\widetilde Z_0 = Y_0$. 
This is mainly due to Levy's characterization of the Brownian motion which is applicable here as for all $x \in \R^d$, $\lambda^2_\delta(x) + \pi_\delta^2(x) = 1$. 

Moreover, that is worth mentioning that there exist two $d$-dimensional Brownian motions $(\beta^0_t)$ and $(\tilde\beta^0_t)$ adapted to the filtration $\mathbb{F}^0$ such that if we define for all $t \geq 0$,  $m^{X, \delta}_t = \mathcal{L}^1(X^\delta_t)$ and $m^{Y, \delta}_t = \mathcal{L}^1(Y^\delta_t)$, we have
	\begin{equation*}
		\mathrm{d}_t  m^{X, \delta}_t = \nabla \cdot \Big(\frac{\sigma_0^2}{2} \nabla m^{X, \delta}_t + m^{X, \delta}_t \big(\nabla V + \alpha\big(m^{X, \delta}_t - \int_{\R^d} x m^{X, \delta}_t(\mathrm{d}x)\big)\big)\Big)\mathrm{d}t  -\sigma_0 \nabla m_t \cdot \mathrm{d}\beta^0_t,
	\end{equation*}
and 
	\begin{equation*}
		\mathrm{d}_t  m^{Y, \delta}_t = \nabla \cdot \Big(\frac{\sigma_0^2}{2} \nabla m^{Y, \delta}_t + m^{Y, \delta}_t \big(\nabla V + \alpha\big(m^{Y, \delta}_t - \int_{\R^d} x m^{Y, \delta}_t(\mathrm{d}x)\big)\big)\Big)\mathrm{d}t  -\sigma_0 \nabla m_t \cdot \mathrm{d}\tilde\beta^0_t.
\end{equation*}

The proof of this result is quite straightforward; see \cite[Vol. II, Thm 4.14]{carmona2018probabilistic}. Even if this may appear very clear, it is important to mention it here. The fact that this approach works is primarily due to the reflection being applied with respect to something measurable with respect to the \textit{common noise}.

~\\

\textit{Step 2:} \textbf{A variance estimate}\\

 Let us now focus on the study of 
\begin{equation*}
	\E_{\mathbb{P}_1}\left[\left| X^\delta_t - \E_{\mathbb{P}_1}[X^\delta_t]\right|^2\right] \text{ and } \E_{\mathbb{P}_1}\left[\left| Y^\delta_t - \E_{\mathbb{P}_1}[Y^\delta_t]\right|^2\right]. 
\end{equation*}
First, for any $t > 0$, we can write: 
\begin{align*}
	\frac{\mathrm{d}}{\mathrm{d}t}\left| X^\delta_t - \E_{\mathbb{P}_1}[X^\delta_t]\right|^2 = & -2\left( X^\delta_t - \E_{\mathbb{P}_1}[X^\delta_t]\right)\cdot\left(\nabla V(X^\delta_t) - \E_{\mathbb{P}_1}[\nabla V(X^\delta_t)]\right) \\ 
	& - 2\alpha \left| X^\delta_t - \E_{\mathbb{P}_1}[X^\delta_t] \right|^2.
\end{align*}
Then, using the fact that $\nabla V$ is Lipschitz continuous, we get
\begin{equation*}
	\frac{\mathrm{d}}{\mathrm{d}t}\E\left[|X_t - \E_{\mathbb{P}^1}[X_t]|^2\right] \leq -2(\alpha - 2L_V)\E\left[|X_t - \E_{\mathbb{P}^1}[X_t]|^2\right].
\end{equation*}
Applying Gronwall Lemma, we obtain 
\begin{equation}\label{eq:proof:identification}
	\E_{\mathbb{P}_1}\left[|X^\delta_t - \E_{\mathbb{P}^1}[X^\delta_t]|^2\right] \leq \E\left[|X_0 - \E_{\mathbb{P}^1}[X_0]|^2\right]\exp\left( -2(\alpha - 2L_V)t\right).
\end{equation}
The same proof stands for $Y^\delta$. 

~\\
~\\
~\\
~\\

\textit{Step 3:} \textbf{Ergodicity for the moment of order one}. \\

The previous step shows that over large time scales, the conditional law of the process $X$ can be identified with its first moment. Consequently, to understand the long-term behavior of the solution $(m_t)_{t\geq 0}$ to Equation \eqref{eq:SFP:sigma0}, it is enough to study the long-term behavior of 
\begin{equation*}
	 \left(\int_{\R^d} x \: m_t(\mathrm{d}x)\right)_{t \geq 0}. 
\end{equation*}
In particular, we would like to show that it admits a unique equilibrium. We are then left with the study of the following dynamic:
\begin{equation*}
	\mathrm{d}\mathbb{E}_{\mathbb{P}^1}[X^\delta_t] = -\mathbb{E}_{\mathbb{P}^1}[\nabla V(X^\delta_t)]\mathrm{d}t + \sigma_0\mathrm{d}B^0_t. 
\end{equation*}
In order to prove that the latter is ergodic as soon as $\sigma_0 > 0$, let us proceed as in \cite{durmus2020elementary}, and introduce the following quantity:
\begin{equation}\nonumber
\begin{aligned}[t]
	& R_0 = \inf\left\{ s \geq 0, \kappa(r) \geq 0, \: \forall r \geq s\right\},\\
	& R_1 = \inf\left\{ s \geq R_0, s(s-R_0)\kappa(r) \geq 4 \sigma_0^2, \: \forall r \geq s\right\}. 
\end{aligned}
\end{equation}
Moreover, we consider $\varphi$, $\Phi$, $g : [0,+\infty) \to [0,+\infty)$ defined by
\begin{equation}\nonumber
\begin{aligned}[t]
	& \varphi(r) = \exp\left(-\frac{1}{2\sigma_0^2}\int_0^r s \kappa_-(s) \mathrm{d}s \right), \\
    & \Phi(r) = \int_0^r \varphi(s) \mathrm{d}s,\\
	& g(r) = 1 - \frac{\ell}{2}\int_0^{r\wedge R_1} \Phi(s)/\varphi(s) \mathrm{d}s,
	%& f(r) = \int_0^r \varphi(s) g(s) \mathrm{d}s,
\end{aligned}
\end{equation}
where $\kappa_- = \max(0, -\kappa)$ and $\ell = \left( \int_0^{R_1} \Phi(s) \varphi(s)^{-1} \mathrm{d}s\right)^{-1}$.  We now define an increasing function $f : [0, +\infty) \to [0,+\infty)$ by:
\begin{equation}\nonumber
	f(r) = \int_0^r \varphi(s)g(s) \mathrm{d}s.
\end{equation}
The function $f$ that has been constructed is positive, non-decreasing, and concave. Moreover, it satisfies
   \begin{equation}\label{ineq:f}
        \varphi(R_0)r/2 \leq f(r) \leq r.
   \end{equation}
   This ensures that $(x,y) \mapsto f(|x-y|)$ defines a distance that is equivalent to the Euclidean one. Below, we will use contraction properties in Wasserstein-1 distance based on the underlying distance $f(|x-y|)$. This contraction property is a consequence of the following inequalities which hold for all $r > 0$:
\begin{equation}\label{contraction}
	f''(r) - \frac{1}{2\sigma_0^2} r\kappa(r) f'(r) \leq -\ell f(r)/2. 
\end{equation}
For a proof of this inequality, see \cite[Sec. 4]{eberle2016reflection}. \\

Now, using Itô's formula, we get that for any $\delta > 0$,
	\begin{align*}
		\mathrm{d}|\mathbb{E}_{\mathbb{P}^1}[X^{\delta}_t] - \mathbb{E}_{\mathbb{P}^1}[Y^{\delta}_t]|^2 = & -2(\mathbb{E}_{\mathbb{P}^1}[X^{\delta}_t] - \mathbb{E}_{\mathbb{P}^1}[Y^{\delta}_t])\cdot(\mathbb{E}_{\mathbb{P}^1}[\nabla V(X^{\delta}_t)] - \mathbb{E}_{\mathbb{P}^1}[\nabla V(Y^{\delta}_t)])\mathrm{d}t \\
		& + 4\sigma_0(\mathbb{E}_{\mathbb{P}^1}[X^{\delta}_t] - \mathbb{E}_{\mathbb{P}^1}[Y^{\delta}_t])\pi_\delta(E^\delta_t)e^\delta_t(e^\delta_t)^T\mathrm{d}B^0_t \\
		& + 4\sigma_0^2\pi_\delta(E^\delta_t)^2(e^\delta_t)^T\mathrm{d}t. 
	\end{align*}
For any $\varepsilon > 0$, let us introduce $\psi_\varepsilon : [0,+\infty] \in r \mapsto (r+\varepsilon)^{1/2}$. This function is continuously twice differentiable, then we can write
\begin{align*}
	\mathrm{d}\psi_{\varepsilon}(|\mathbb{E}_{\mathbb{P}^1}[X^{\delta}_t] - \mathbb{E}_{\mathbb{P}^1}[Y^{\delta}_t]|^2) = & -2 \psi^\prime_{\varepsilon}(|E^\delta_t|^2) E^\delta_t\cdot(\mathbb{E}_{\mathbb{P}^1}[\nabla V(X^{\delta}_t)] - \mathbb{E}_{\mathbb{P}^1}[\nabla V(Y^{\delta}_t)])\mathrm{d}t \\
	& + 4\sigma_0 \psi^\prime_{\varepsilon}(|E^\delta_t|^2) E^\delta_t \pi_\delta(E^\delta_t)e^\delta_t(e^\delta_t)^T\mathrm{d}B^0_t \\
	& + 4 \sigma_0^2\psi^\prime_{\varepsilon}(|E^\delta_t|^2)\pi_\delta(E^\delta_t)^2(e_t^\delta)^T\mathrm{d}t \\
	& + 8\sigma_0^2\psi^{\prime\prime}_{\varepsilon}(|E^\delta_t|^2)|E^\delta_t|^2\psi_\delta(E^\delta_t)(e_t^\delta)^T\mathrm{d}t. 
\end{align*}
We now want to take the limit $\varepsilon\to 0$. Using dominated convergence theorem and stochastic dominated convergence theorem as stated in \cite[Chap. IV, Sec. 3]{revuz2013continuous}, combined to the fact that $4r\psi^\prime_{\varepsilon}(r^2) \leq 1$, we can deal with the first two lines and get that for all $t\geq 0$
\begin{equation}
\begin{aligned}\label{eq:TCD1}
	\lim_{\varepsilon \to 0}\int_0^t&  2 \psi^\prime_{\varepsilon}(|E^\delta_s|^2) E^\delta_s\cdot(\mathbb{E}_{\mathbb{P}^1}[\nabla V(X^{\delta}_s)] - \mathbb{E}_{\mathbb{P}^1}[\nabla V(Y^{\delta}_s)])\mathrm{d}s \\
	& = \int_0^t e^\delta_s\cdot(\mathbb{E}_{\mathbb{P}^1}[\nabla V(X^{\delta}_s)] - \mathbb{E}_{\mathbb{P}^1}[\nabla V(Y^{\delta}_s)])\mathrm{d}s   
\end{aligned}
\end{equation}
and 
\begin{equation}
\begin{aligned}\label{eq:TCD2}
	\lim_{\varepsilon \to 0} \int_0^t  4\sigma_0 \psi^\prime_{\varepsilon}(|E^\delta_s|^2) E^\delta_s \pi_\delta(E^\delta_s)e^\delta_s(e^\delta_s)^T\mathrm{d}B^0_s 
	 =  \int_0^t 4\sigma_0 \pi_\delta(E^\delta_s)(e^\delta_s)^T\mathrm{d}B^0_s
\end{aligned}
\end{equation}
For the two last, we need to take advantage of the presence of the function $\pi_\delta$ for $\delta > 0$. We have  
\begin{align*}
	& \left|4 \sigma_0^2\psi^\prime_{\varepsilon}(|E^\delta_t|^2)\pi_\delta(E^\delta_t)^2(e_t^\delta)^T+ 8\sigma_0^2\psi^{\prime\prime}_{\varepsilon}(|E^\delta_t|^2)|E^\delta_t|^2\psi_\delta(E^\delta_t)(e_t^\delta)^T\right| \\
	& \leq \left|\pi_\delta(E^\delta_t)^2\sigma_0^2\left( 4\psi^\prime_{\varepsilon}(|E^\delta_t|^2) + 8\psi^{\prime\prime}_{\varepsilon}(|E^\delta_t|^2)|E^\delta_t|^2\right)\right|. 
\end{align*}
Moreover, we know that $\psi^\prime_{\varepsilon}(r^2) + 2\psi^{\prime\prime}_{\varepsilon}(r^2)r^2 \leq r^{-3}$, for all $r > 0$ and $\varepsilon \leq 1$. Using the presence of $\pi_\delta$, we have that the integrand is null near 0. Then, we can once again apply the dominated convergence theorem and obtain
\begin{equation}\label{eq:TCD3}
	\lim_{\varepsilon \to 0} \int_0^t 4 \left\{\sigma_0^2\psi^\prime_{\varepsilon}(|E^\delta_s|^2)\pi_\delta(E^\delta_s)^2(e_s^\delta)^T+ 8\sigma_0^2\psi^{\prime\prime}_{\varepsilon}(|E^\delta_s|^2)|E^\delta_s|^2\psi_\delta(E^\delta_s)(e_s^\delta)^T \right\}\mathrm{d}s = 0. 
\end{equation}
Finally, combining Equations \eqref{eq:TCD1}, \eqref{eq:TCD2} and \eqref{eq:TCD3}, we get 
\begin{align}\label{eq:LT}
	\mathrm{d}| \mathbb{E}_{\mathbb{P}^1}[X^{\delta}_t] - \mathbb{E}_{\mathbb{P}^1}[Y^{\delta}_t]| & = - e^{\delta}_t\cdot\Big(\mathbb{E}_{\mathbb{P}^1}[\nabla V(X^{\delta}_t)] - \mathbb{E}_{\mathbb{P}^1}[\nabla V(Y^{\delta}_t)]\Big) \mathrm{d}t \\
	& \nonumber + 2\sigma_0\pi_{\delta}(E^{\delta}_t)( e^{\delta}_t)^T \mathrm{d}B^0_t. 
\end{align}

Using Equation \eqref{eq:LT}, we can write:
\begin{align*}
	\mathrm{d}f(|E^\delta_t|) & = -f^\prime( | E^{\delta}_t|) e^{\delta}_t\cdot\Big( \mathbb{E}_{\mathbb{P}^1}[\nabla V(X^{\delta}_t)] - \mathbb{E}_{\mathbb{P}^1}[\nabla V(Y^{\delta}_t)]\Big) \mathrm{d}t  \\
	& + 2\sigma_0f^\prime( |E^{\delta}_t|)\pi_{\delta}( E^{\delta}_t)(e^{\delta}_t)^t \mathrm{d}B^0_t \\
	& + 2\sigma_0^2 f^{\prime\prime}( | E^{\delta}_t|)\pi_{\delta}(E^{\delta}_t)^2 \mathrm{d}t.
\end{align*}
To control the first term, we carry out the following decomposition
\begin{align*}
	-f^\prime( | E^{\delta}_t|) e^{\delta}_t\cdot &\big( \mathbb{E}_{\mathbb{P}^1}[\nabla V(X^{\delta}_t)] - \mathbb{E}_{\mathbb{P}^1}[\nabla V(Y^{\delta}_t)]\big)  \\
	= & -f^\prime( | E^{\delta}_t|) e^{\delta}_t\cdot\big( \mathbb{E}_{\mathbb{P}^1}[\nabla V(X^{\delta}_t)] - \nabla V(\mathbb{E}_{\mathbb{P}^1}[X^{\delta}_t])\big)  \\
	& - f^\prime( | E^{\delta}_t|) e^{\delta}_t\cdot\big( \nabla V(\mathbb{E}_{\mathbb{P}^1}[X^{\delta}_t]) - \nabla V(\mathbb{E}_{\mathbb{P}^1}[Y^{\delta}_t])\big)  \\
	& -f^\prime( | E^{\delta}_t|) e^{\delta}_t\cdot\big( \mathbb{E}_{\mathbb{P}^1}[\nabla V(Y^{\delta}_t)] - \nabla V(\mathbb{E}_{\mathbb{P}^1}[Y^{\delta}_t])\big).
\end{align*}
Taking the expectation and leveraging on the fact that $f'$ is globally bounded, the existence of $C_1 > 0$, independent of $t$ and $\delta$, such that
\begin{equation*}
\left\{ \begin{array}{ll}
 -f^\prime( | E^{\delta}_t|) e^{\delta}_t\cdot\big( \mathbb{E}_{\mathbb{P}^1}[\nabla V(X^{\delta}_t)] - \nabla V(\mathbb{E}_{\mathbb{P}^1}[X^{\delta}_t])\big) \leq C_1\E\left[ |X^{\delta}_t - \mathbb{E}_{\mathbb{P}^1}[X^{\delta}_t]|\right]\\
 ~\\
	-f^\prime( | E^{\delta}_t|) e^{\delta}_t\cdot\big( \mathbb{E}_{\mathbb{P}^1}[\nabla V(Y^{\delta}_t)] - \nabla V(\mathbb{E}_{\mathbb{P}^1}[Y^{\delta}_t])\big) \leq C_1\E\left[ |Y^{\delta}_t - \mathbb{E}_{\mathbb{P}^1}[Y^{\delta}_t]|\right]
	\end{array}
	\right.
\end{equation*}
Moreover, using Step 1, we obtain
\begin{align*}
	& -\mathbb{E}\big[f^\prime( | E^{\delta}_t|) e^{\delta}_t\cdot\big( \mathbb{E}_{\mathbb{P}^1}[\nabla V(X^{\delta}_t)] - \mathbb{E}_{\mathbb{P}^1}[\nabla V(Y^{\delta}_t)]\big)\big]  \\
	& \leq \: C_2e^{-( \alpha-2L_V)t}  - \E\big[ f^\prime( | E^{\delta}_t|) e^{\delta}_t\cdot (\nabla V( \mathbb{E}_{\mathbb{P}^1}[X_t^\delta]) - \nabla V(\mathbb{E}_{\mathbb{P}^1}[Y_t^\delta])) \big].
\end{align*}
Now, using the contraction property of $f$ stated in Equation \eqref{contraction}, we obtain
\begin{equation}\label{eq:contract}
	\frac{\mathrm{d}}{\mathrm{d}t}\E[ f( |E^\delta_t|)] \leq C_1e^{-(\alpha - 2L_V) t} -\ell\sigma_0^2f( |E^{\delta}_t|) + \ell\sigma_0^2\delta +|\kappa_-|_{\infty}\delta.
\end{equation}
Using Gronwall Lemma in Equation \eqref{eq:contract}, we get
 \begin{equation}\label{eq:ergo}
 	\E[ f( |\E_{\mathbb{P}^1}[X_t] - \mathbb{E}_{\mathbb{P}^1}[Y_t]|)] \leq C_3(e^{-(\alpha - 2L_V) t} + e^{-\ell\sigma_0^2t} + h(\delta)), 
 \end{equation}
for some $C_3$ that only depends on $P_0$ and $Q_0$, and where $h(\delta) =  \ell\sigma_0^2\delta +|\kappa_-|_{\infty}\delta \to 0$, whenever $\delta \to 0$.  

~\\

\textit{Step 4:} \textbf{Uniqueness recovery}\\

The proof of Theorem \ref{th:sg0} is now quite straightforward. Taking advantage of the previous sections, we can perform the following decomposition:
	\begin{align*}
		\mathrm{d}_1^{\mathcal{P}(\R^d)}( P_t, Q_t) & \leq \E_{\mathbb{P}_0}[\mathrm{d}_1^{\R^d}(m^X_t, m^Y_t)] \\
		& \leq \E[|X^\delta_t - Y^\delta_t|] \\
		& \leq \E[|X^\delta_t - \E_{\mathbb{P}_1}[X^\delta_t]|] + \E[|Y^\delta_t - \E_{\mathbb{P}_1}[Y^\delta_t]|] + |\E[|\E_{\mathbb{P}_1}[X^\delta_t] - \E_{\mathbb{P}_1}[Y^\delta_t]|]|
	\end{align*}
Using the result of Step 1, we obtain
\begin{align*}
	 & \E[|X^\delta_t - \E_{\mathbb{P}_1}[X^\delta_t]|] + \E[|Y^\delta_t - \E_{\mathbb{P}_1}[Y^\delta_t]|]  \\
	 & \leq \left(\E\left[|X_0 - \E_{\mathbb{P}^1}[X_0]|^2\right]^{1/2} + \E\left[|Y_0 - \E_{\mathbb{P}^1}[Y_0]|^2\right]^{1/2}\right)\exp\left( -(\alpha - 2C_V)t\right). 
\end{align*}
Moreover, thanks to Equation \eqref{eq:ergo}
\begin{align*}
	|\E[|\E_{\mathbb{P}_1}[X^\delta_t] - \E_{\mathbb{P}_1}[Y^\delta_t]|]| \leq C_3(e^{-(\alpha - 2L_V) t} + e^{-\ell\sigma_0^2t} + h(\delta))
\end{align*}
Finally, we obtain the existence of $C > 0$ depending only on $P_0$ and $Q_0$ such that,   
\begin{equation*}
	\mathrm{d}_1^{\mathcal{P}(\R^d)}( P_t, Q_t) \leq C\left( e^{-(\alpha - 2L_V) t} + e^{-\ell\sigma_0^2t} + h(\delta)\right).
\end{equation*}
Taking $\delta \to 0$ gives the expected result.
 \end{proof}

 \subsection*{Acknowledgments} The author would like to thank Pierre Cardaliaguet for suggesting the problem and fruitful discussions during the preparation of this work.
	
% BIBLIO

	% BIBLIO
	
	\bibliographystyle{plain} 
	\bibliography{LongTimeMKVCN}  

		%Appendix
		\appendix
	\section{}\label{Appendix}

	%Proof of Proposition 1

	\subsection{Proof of Proposition \ref{Lacker}}\label{App:A0}
Equation \eqref{eq:equationP} comes from Section 1.5 in \cite{lacker2020superposition} (Equation (1.15)). Hence, if $(m_t)_{t\geq 0}$ is an invariant measure in the sense of Definition \ref{D2}, then $\bar P$ straightforwardly is a solution of Equation \eqref{eq:InvMeasure}. Conversely, let us consider a probability measure $\bar P$, solution of Equation \eqref{eq:InvMeasure}. More precisely, for all $F \in \mathcal{C}_b^2(\mathcal{P}(\R^d))$
\begin{align*}
	\left\langle \widebar P, \mathcal{M} F\right\rangle = 0.
\end{align*}
Let us consider $T > 0$, then
\begin{align*}
	&\int_{\mathcal{P}(\R^d)}\| \nabla V + \nabla W\ast m\|_{L^2(m)}^2 \bar P(\mathrm{d}m) \\
	& \leq C\Big(1 + \int_{\mathcal{P}(\R^d)}\int_{\R^d} |\nabla V(x)|^2 m(\mathrm{d}x)\bar P(\mathrm{d}m) + \int_{\mathcal{P}(\R^d)}\int_{\R^d} |\nabla W\ast m(x)|^2 m(\mathrm{d}x)\bar P(\mathrm{d}m)\Big) \\
	& \leq C\Big(1 + \int_{\mathcal{P}(\R^d)}\int_{\R^d} |x|^2 m(\mathrm{d}x)\bar P(\mathrm{d}m) + \int_{\mathcal{P}(\R^d)} |\nabla W\ast m(0)|^2\bar P(\mathrm{d}m)\Big)\\
	& \leq C\Big(1 + \int_{\mathcal{P}(\R^d)}\int_{\R^d} |x|^2 m(\mathrm{d}x)\bar P(\mathrm{d}m)\Big),
\end{align*}
for some constant $C$ that may change from line to line and depends only on the Lipschitz constant of $\nabla V$ and $\nabla W$. Then using Proposition \ref{T1}, we get that 
\begin{align*}
	\int_{\mathcal{P}(\R^d)}\| \nabla V + \nabla W\ast m\|_{L^2(m)}^2 \bar P(\mathrm{d}m) < +\infty.
\end{align*}
Applying Theorem 1.5 in \cite{lacker2020superposition}, we get the existence of a process $(\mu_t) \in \mathcal{C}([0,T], \mathcal{P}(\R^d))$ such that $\mu$ has dynamic given by \eqref{eq:main} and for all $t \in [0,T]$, $\mathcal{L}(\mu_t) = \bar P$. Then by weak uniqueness of the solutions of \eqref{eq:main}, we get that $\bar P$ is an invariant measure for $m$ in the sense of Definition \ref{D2}.

	%Proof of lemma 2
	\subsection{Proof of Lemma \ref{le:L2}}\label{App:A1}
We consider the process $(X_t)$, driven by the dynamic \eqref{LP} and with initial condition $X_0$ such that $\mathcal{L}(X_0) = P_0$, in the sense of Definition \ref{D1}. 	
Let us denote in the following $m_2(t) = \mathbb{E}\left[ |X_t|^2\right]$. Then expanding using the Ito formula and taking the time derivative, gives:
    \begin{equation}\nonumber
    \begin{aligned}[t]
        m_2'(t) & = -2\mathbb{E}\left[ X_t \cdot \left( \nabla V(X_t) + \nabla W \ast m_t(X_t)\right) \right] + (\sigma^2 +\sigma_0^2)d\\
        & = -2\mathbb{E}\left[ X_t \cdot \left( \nabla V(X_t) - \nabla V(0)\right)\right] \\
        & \vspace{4cm} - 2\nabla V(0)\mathbb{E}\left[ X_t\right] - 2\mathbb{E}\left[ X_t \cdot \nabla W\ast m_t(X_t) \right] + (\sigma_0^2 + \sigma^2)d.
    \end{aligned}
    \end{equation}
Moreover, if $\widetilde X_t$ an independent copy of $X_t$, then:
\begin{equation}\nonumber
\begin{aligned}[t]
    \nabla W \ast m_t(X_t) &= \int_{\R^d} \nabla W(X_t - y) m_t(\mathrm{d}y) \\
    & =  \mathbb{E}[ \nabla W( X_t - \widetilde X_t) | X_t, \mathcal{F}^0_0].
\end{aligned}
\end{equation}
Now, the fact that $W$ is symetric gives $2\mathbb{E}[X_t\nabla W\ast m_t(X_t)] = \mathbb{E}[ ( X_t - \widetilde X_t)\cdot\nabla W(X_t - \widetilde X_t)]$. This decomposition is the key, the end of the proof is straightforward using Assumption made on the potential $W$ (see Assumptions \ref{A2}).

	%A measurable selection result
	\subsection{A measurable selection result}\label{appen:SM}

In this part, we will mainly prove Lemma \ref{MS}:
\begin{lemma}
    Let $m$ and $\rho$ be two probability measure valued random variables which are $\mathcal{F}^0_0$-measurable. Assuming that both random measures $m$ and $\rho$ admit a moment of order two almost surely, then there exists a random variable $\xi$, such that almost surely:
    \begin{equation}\nonumber
        \xi \in \argmin_{\pi \in \Pi\left( m, \rho\right)} \int_{\R^d}|x-y|^2 \pi(\mathrm{d}x, \mathrm{d}y).
    \end{equation}
\end{lemma}

\begin{proof}
    The purpose of the proof is to show that there exists a measurable function $\phi : \mathcal{P}(\R^d) \times \mathcal{P}(\R^d) \to \mathcal{P}(\R^d)$, such that for all $(m,\rho) \in \mathcal{P}(\R^d) \times \mathcal{P}(\R^d) \to \mathcal{P}(\R^d)$, $\phi(m,\rho) \in \Pi_{\text{opt}}(m,\rho)$, the set of minimizers for the transport problem. First of all, it is known that $\Pi_{\text{opt}}$ is never empty for the quadratic cost, see \cite[Chap. 4, Thm 4.1]{villani2009optimal}. Then, we need to show that the set-valued function
    \begin{equation}\nonumber
    \Phi: 
        \begin{array}{ll}
            \mathcal{P}(\R^d) \times \mathcal{P}(\R^d) \to 2^{\mathcal{P}(\R^d)} \\
            (m, \rho) \mapsto \Pi(m, \rho),
        \end{array}
    \end{equation}
    is measurable, where $2^A$ stands for the set of subsets of $A$. Considering the graph $\Gamma_{\Phi}$, of $\Phi$:
    \begin{equation}\nonumber
        \Gamma_{\Phi} = \left\{ \left( (m,\rho), \xi\right), \: \xi \in \Phi(m,\rho)\right\},
    \end{equation}
    it is clear that it is a closed set, and then measurable. In particular, the multi-application $\Phi$ is measurable.
    Moreover, by \cite[Lemma 12.1.7]{stroock1997multidimensional}, the application $\Psi$ which associate to any compact set $K \in \mathcal{P}(\R)^2$ the set
\begin{align*}
	\Psi(K) = \arg\inf_{\pi \in K}\int_{\R^2} |x-y| \pi(dx,dy) 
\end{align*}
is measurable. Then the function 
    \begin{equation}\nonumber
    \Lambda: 
        \begin{array}{ll}
            \mathcal{P}(\R^d) \times \mathcal{P}(\R^d) \to \mathrm{Ens}(\mathcal{P}(\R^d)) \\
            (m, \rho) \mapsto \Pi_{\text{opt}}(m, \rho),
        \end{array}
    \end{equation}
    is a compound of two measurable functions. Finally, by the measurable selection theorem, there exists a measurable function $\phi$ such that for all $(m,\rho) \in \mathcal{P}(\R^d) \times \mathcal{P}(\R^d)$, $\phi(m, \rho) \in \Pi_{\text{opt}}(m, \rho)$. \end{proof}

\end{document}